\documentclass{scrartcl}
\usepackage[T1]{fontenc}
\usepackage[utf8]{inputenc}
\usepackage{amsmath, amsthm, amssymb}
\usepackage{dsfont}
\usepackage{graphicx}
\usepackage[format=plain]{caption}
\usepackage{capt-of}
\usepackage{subcaption}
\usepackage{bm}
\usepackage{tikz}
\usetikzlibrary{arrows,shapes,trees,calc,through}

\newcommand{\norm}[1]{\left\| #1 \right\|}  
\newcommand{\scprd}[1]{\left\langle #1 \right\rangle}  
\renewcommand{\d}{\,\mathrm{d}} 
\newcommand{\e}{\mathrm{e}} 
\newcommand{\N}{\mathbb{N}}  

\newcommand{\R}{\mathbb{R}}

\newcommand{\supp}{\operatorname{supp}}

\renewcommand{\phi}{\varphi}
\newcommand{\ul}{\underline}
\newcommand{\ol}{\overline}

\renewcommand{\Re}{\operatorname{Re}}

\numberwithin{equation}{section}

\newtheorem{thm}{Theorem}[section]
\newtheorem{cor}[thm]{Corollary}

\newtheorem{lm}[thm]{Lemma}
\newtheorem{cond}[thm]{Condition}

\theoremstyle{definition} \newtheorem{ex}[thm]{Example}
\theoremstyle{definition}

\title{Well-posedness and stability of non-autonomous semilinear input-output 
systems}

\author{Jochen Schmid\\  
\small Fraunhofer Institute for Industrial Mathematics (ITWM), 67663 Kaiserslautern, Germany\\ 
\small jochen.schmid@itwm.fraunhofer.de}

\date{}

\begin{document}

\maketitle

\begin{abstract}
\small{ \noindent 
We establish well-posedness results for non-autonomous semilinear input-output systems, the central assumption being the scattering-passivity of the considered semilinear system. Along the way, we also establish global stability estimates. We consider both systems with distributed control and observation and systems with boundary control and observation. Applications are given to nonlinearly controlled collocated systems and to nonlinearly controlled port-Hamiltonian systems. 
}
\end{abstract}


\section{Introduction}

In this paper, we establish well-posedness results for non-autonomous semilinear input-output systems whose input and output operators are linear. 
%
We consider semilinear systems with bounded control and observation operators described by 
\begin{equation} \label{eq:wp-semilin-distr-contr/obs}
\begin{gathered}
\dot{x}(t) = \mathcal{A}(t)x(t) + f(t,x(t)) + \mathcal{B}(t)u(t)\\
y(t) = \mathcal{C}(t)x(t),
\end{gathered}
\end{equation}
and semilinear systems with unbounded control and observation operators described by
\begin{equation} \label{eq:wp-semilin-bdry-contr/obs}
\begin{gathered}
\dot{x}(t) = \mathcal{A}(t)x(t) + f(t,x(t)) \\
u(t) = \mathcal{B}(t)x(t) \quad \text{and} \quad y(t) = \mathcal{C}(t)x(t).
\end{gathered}
\end{equation}
In these equations, $x(t) \in X$ is the state of the system at time $t$ ($X$ being a Banach space) and $u$, $y$ are the control input and observation output of the system taking values in an input-value and an output-value space $U$, $Y$ (Banach spaces) respectively. Also, 
$$\mathcal{A}(t): D(\mathcal{A}(t)) \subset X \to X$$
is a linear operator and $f: \R^+_0 \times X \to X$ is a time-dependent nonlinearity. And finally, in the case~\eqref{eq:wp-semilin-distr-contr/obs} of bounded control and observation operators, the input and output operators 
\begin{align} \label{eq:bd-io-op}
\mathcal{B}(t): U \to X \quad \text{and} \quad \mathcal{C}(t): X \to Y
\end{align} 
are bounded linear operators while, in the case~\eqref{eq:wp-semilin-bdry-contr/obs} of unbounded control and observation operators, the input and output operators 
\begin{align} \label{eq:unbd-io-op}
\mathcal{B}(t): D(\mathcal{B}(t)) \subset X \to U \quad \text{and} \quad \mathcal{C}(t): D(\mathcal{C}(t)) \subset X \to Y
\end{align}
are unbounded linear operators. In applications, distributed control and observation correspond to bounded in- and output operators, while boundary control and observation typically -- but not necessarily -- correspond to unbounded in- and output operators. (See the example in Section~\ref{sect:appl-bd} below, for instance, for a boundary control and observation system that can be formulated with bounded in- and output operators.) 
\smallskip 

What we are interested in here is the well-posedness of non-autonomous semilinear systems as above. In rough terms, this means that for every initial state $x_0 \in X$ and every input $u \in L^2_{\mathrm{loc}}(\R^+_0,U)$ the respective system has a unique generalized solution $x(\cdot,x_0,u) \in C(\R^+_0,X)$ and a unique generalized output $y(\cdot,x_0,u) \in L^2_{\mathrm{loc}}(\R^+_0,Y)$ and that these quantities depend continuously on $(x_0,u) \in X\times L^2_{\mathrm{loc}}(\R^+_0,U)$. 
%
Well-posedness is a most fundamental notion in control theory and is relevant for almost every control system. 
Accordingly, there exist a lot of papers devoted partly or completely to the well-posedness of input-output systems, especially infinite-dimensional ones. 
%
So far, however, most of these papers on infinite-dimensional systems have been confined 
\begin{itemize}
\item either to the case of linear non-autonomous systems (like~\cite{Sc02}, \cite{Ha06}, \cite{ScWe10}, \cite{ChWe15} or \cite{JaLa19}, for instance)
\item or to the case of autonomous semilinear systems (like~\cite{TuWe14} or \cite{HaCaZw19}, for instance).
\end{itemize}
In fact, we are aware of only a few papers which treat non-autonomous semilinear systems in the context of control theory, namely~\cite{JaDrPr95}, \cite{BoIdMa06}, \cite{BoId08}. Yet, the settings of those papers are quite different from ours. 
Indeed, \cite{JaDrPr95}, for instance, deal with closed-loop systems arising by nonlinear output feedback from a non-autonomous linear system.  
In particular, 
\cite{JaDrPr95} consider no external inputs $u$ like we do in~\eqref{eq:wp-semilin-distr-contr/obs} and~\eqref{eq:wp-semilin-bdry-contr/obs}. And accordingly, the question of continuous dependence of the solutions and outputs on external inputs and intial states  (which is central to well-posedness and the present paper) does not arise in~\cite{JaDrPr95}. 
\smallskip

Along with well-posedness, we also establish the uniform global stability of the systems~\eqref{eq:wp-semilin-distr-contr/obs} and~\eqref{eq:wp-semilin-bdry-contr/obs}. In fact, uniform global stability estimates for classical solutions of~\eqref{eq:wp-semilin-distr-contr/obs} and~\eqref{eq:wp-semilin-bdry-contr/obs} are also an important intermediate step in proving our well-posedness results. In rough terms, the uniform global stability of~\eqref{eq:wp-semilin-distr-contr/obs} or~\eqref{eq:wp-semilin-bdry-contr/obs} respectively means that $x_0 := 0$ is a globally stable equilibrium point of the system~\eqref{eq:wp-semilin-distr-contr/obs} or~\eqref{eq:wp-semilin-bdry-contr/obs} without external input (that is, with $u := 0$) and that this stability property is affected only slightly by small non-zero inputs $u$. In applications, the external input $u$ typically has the interpretation of a disturbance signal. 
Stability properties like uniform global stability are, of course, already interesting in themselves, but they also play an important role in establishing input-to-state stability. See \cite{MiWi16a}, \cite{Sc18-wISS}, for instance, where the case of autonomous systems is treated. 
\smallskip

In the conference version~\cite{ScDaJaLa19} of this paper, we present only special cases 
of our well-posedness results with technically simpler assumptions like, for instance, a time-independent domain assumption on the linear parts $\mathcal{A}(t)$ of~\eqref{eq:wp-semilin-distr-contr/obs} or~\eqref{eq:wp-semilin-bdry-contr/obs} respectively. Also, \cite{ScDaJaLa19} does not contain proofs but only rough sketches of proofs and, moreover, \cite{ScDaJaLa19} does not give any applications of the abstract well-posedness and stabiliy results.
\smallskip

In the present paper, we proceed as follows. Section~\ref{sect:preliminaries} provides the precise definitions of well-posedness and uniform global stability and, moreover, establishes important lemmas on non-autonomous semilinear systems without control or observation. 
In Section~\ref{sect:wp-bd-io-op} and~\ref{sect:wp-unbd-io-op}, we establish our well-posedness and uniform global stability results for non-autonomous semilinear systems~\eqref{eq:wp-semilin-distr-contr/obs} and~\eqref{eq:wp-semilin-bdry-contr/obs} with bounded or unbounded control and observation operators, respectively. A central assumption for the proof of those results is the scattering passivity of~\eqref{eq:wp-semilin-distr-contr/obs} or~\eqref{eq:wp-semilin-bdry-contr/obs} respectively. In our exposition, we make an effort to 
emphasize the parellelism between the cases~\eqref{eq:wp-semilin-distr-contr/obs} and~\eqref{eq:wp-semilin-bdry-contr/obs}. 
As a consequence, the results of Section~\ref{sect:wp-bd-io-op} and Section~\ref{sect:wp-unbd-io-op} look very similar -- but their proofs are different due to the mathematically fairly different situations~\eqref{eq:wp-semilin-distr-contr/obs} and~\eqref{eq:wp-semilin-bdry-contr/obs} of bounded or unbounded control and observation operators, respectively. 
In Section~\ref{sect:applications}, we finally apply the abstract well-posedness and stability results to two classes of non-autonomous semilinear systems, namely those arising by coupling a nonlinear controller by standard feedback interconnection to a linear collocated system (Section~\ref{sect:appl-bd}) or to a linear port-Hamiltonian system (Section~\ref{sect:appl-unbd}). Concrete examples include nonlinearly controlled Euler-Bernoulli and Timoshenko beams with possibly time-dependent material parameters (like the flexural rigidity and the mass density of the beams, for instance).  
\smallskip

In the entire paper, $\R^+_0 := [0,\infty)$ denotes the non-negative reals and, as usual, $\mathcal{K}$ and $\mathcal{K}_{\infty}$ denote the following classes of comparison functions:
\begin{equation} \label{eq:comparison-fcts-def}
\begin{gathered}
\mathcal{K} := \{ \gamma \in C(\R^+_0,\R^+_0): \gamma \text{ strictly increasing with } \gamma(0) = 0 \} \\
\mathcal{K}_{\infty} := \{ \gamma \in \mathcal{K}: \gamma \text{ unbounded} \}.
\end{gathered}
\end{equation}
Also, $C_c^2(\R^+_0,U)$ denotes the space of $C^2(\R^+_0,U)$-functions $u$ with compact support in $\R^+_0$ 
and for $u \in L_{\mathrm{loc}}^p(\R^+_0,U)$ we will use the following short-hand notations: 
\begin{align} \label{eq:p-norms-abbrev}
\norm{u}_{[0,t],p} := \norm{u|_{[0,t]}}_{L^p([0,t],U)}.
\end{align}  
And finally, for a linear operator $\mathcal{A}$ the symbol $D(\mathcal{A})$ denotes its domain and $L(X,Y)$ denotes the space of bounded linear operators between $X$ and $Y$ with the usual short-hand notation $L(X) := L(X,X)$.

\section{Some preliminaries} \label{sect:preliminaries}

\subsection{Solution concepts and well-posedness}

We begin by properly defining the well-posedness of semilinear systems of the kind~\eqref{eq:wp-semilin-distr-contr/obs} and~\eqref{eq:wp-semilin-bdry-contr/obs} with bounded or unbounded in- and output operators~\eqref{eq:bd-io-op} or~\eqref{eq:unbd-io-op}, respectively. In order to do so, we need various solution concepts.
A \emph{classical solution} to~\eqref{eq:wp-semilin-distr-contr/obs} or~\eqref{eq:wp-semilin-bdry-contr/obs} for given initial state $x_s \in X$ at time $s \in \R^+_0$ and given input $u \in L^2_{\mathrm{loc}}(\R^+_0,U)$ is a function $x \in C^1(J,X)$ on some interval $J$ with $\min J = s$ such that $x(s) = x_s$ and such that 
\begin{itemize}
\item $x(t) \in D(\mathcal{A}(t))$ and~\eqref{eq:wp-semilin-distr-contr/obs} is satisfied for every $t \in J$, or
\item $x(t) \in D(\mathcal{A}(t)) \cap D(\mathcal{B}(t)) \cap D(\mathcal{C}(t))$ and~\eqref{eq:wp-semilin-bdry-contr/obs} is satisfied for every $t \in J$, 
\end{itemize}
respectively. 
A \emph{generalized solution} and a \emph{generalized output} to~\eqref{eq:wp-semilin-distr-contr/obs} or~\eqref{eq:wp-semilin-bdry-contr/obs} for given initial state $x_0 \in X$ at time $0$ and given input $u \in L^2_{\mathrm{loc}}(\R^+_0,U)$ is a function 
$$x \in C(\R^+_0,X) \quad \text{and} \quad y \in L^2_{\mathrm{loc}}(\R^+_0,Y)$$
such that there exists a sequence of 
initial states and inputs $(x_{0 n},u_n) \in X \times C_c^2(\R^+_0,U)$ which converge to $(x_0,u)$:
\begin{align} \label{eq:wp-def-data}
(x_{0 n},u_n) \underset{X \times L^2_{\mathrm{loc}}(\R^+_0,U)}{\longrightarrow} (x_0,u),
\end{align}
and for which the system~\eqref{eq:wp-semilin-distr-contr/obs} or~\eqref{eq:wp-semilin-bdry-contr/obs} 
has a unique global classical solution $x(\cdot,x_{0n},u_n)$ satisfying
\begin{align} \label{eq:wp-def}
x(\cdot,x_{0n},u_n) \underset{C(\R^+_0,X)}{\longrightarrow} x 
\quad \text{and} \quad
\mathcal{C}(\cdot) x(\cdot,x_{0n},u_n) \underset{L^2_{\mathrm{loc}}(\R^+_0,Y)}{\longrightarrow} y.
\end{align}
All three convergences above are w.r.t.~the canonical locally convex topologies, that is, the first convergence~\eqref{eq:wp-def-data} means that $x_{0n} \longrightarrow x_0$ in the norm of $X$ and $\norm{u_n - u}_{[0,t],2} \longrightarrow 0$ as $n \to \infty$ 
for every $t \in (0,\infty)$, while the second and third convergence~(\ref{eq:wp-def}.a) and~(\ref{eq:wp-def}.b) mean that 
\begin{align*}
\norm{x_n - x}_{[0,t],\infty}, \quad \norm{y_n - y}_{[0,t],2} \longrightarrow 0 \qquad (n \to \infty)
\end{align*}
for every $t \in (0,\infty)$, where we used the abbreviations $x_n := x(\cdot,x_{0n},u_n)$ and $y_n := \mathcal{C}(\cdot) x(\cdot,x_{0n},u_n)$ and the notation~\eqref{eq:p-norms-abbrev}.
\emph{Well-posedness} of the system~\eqref{eq:wp-semilin-distr-contr/obs} or~\eqref{eq:wp-semilin-bdry-contr/obs} now  means that, 
for every initial state $x_0 \in X$ and every input $u \in L^2_{\mathrm{loc}}(\R^+_0,U)$, 
the system 
has a unique 
generalized solution and generalized output
\begin{align*}
x(\cdot,x_0,u) \in C(\R^+_0,X) 
\quad \text{and} \quad
y(\cdot,x_0,u) \in L^2_{\mathrm{loc}}(\R^+_0,Y)
\end{align*}
respectively, and that these quantities depend continuously on $(x_0,u)$, that is, 
the functions 
\begin{align*}
(x_0,u) \mapsto x(\cdot,x_0,u) \in C(\R^+_0,X) 
\qquad \text{and} \qquad
(x_0,u) \mapsto y(\cdot,x_0,u) \in L^2_{\mathrm{loc}}(\R^+_0,Y)
\end{align*}
are continuous w.r.t.~the canonical locally convex topologies. 
(A reader familiar with the conference version~\cite{ScDaJaLa19} of this paper might have realized that we slightly modified -- in fact, rectified -- the definition of generalized solutions and outputs from~\cite{ScDaJaLa19}. Indeed, in contrast to that paper, we require here that the approximating inputs $u_n$ from~\eqref{eq:wp-def-data} and~\eqref{eq:wp-def} belong to $C_c^2(\R^+_0,U)$. We do so because otherwise it is not clear how to establish the uniqueness of generalized solutions and outputs under the assumptions of our well-posedness theorems below. See the remarks in parentheses after~\eqref{eq:generalized-sol,distr} and~\eqref{eq:generalized-sol,bdry} below.)

\subsection{Stability}

System~\eqref{eq:wp-semilin-distr-contr/obs} or~\eqref{eq:wp-semilin-bdry-contr/obs} respectively is called \emph{uniformly globally stable} iff for every initial state $x_0 \in X$ at time $0$ and for every input $u \in L^2_{\mathrm{loc}}(\R^+_0,U)$ it has a unique generalized solution $x(\cdot,x_0,u) \in C(\R^+_0,X)$ and there exist comparison functions $\sigma, \gamma \in \mathcal{K}$ as in~\eqref{eq:comparison-fcts-def} such that the estimate
\begin{align} \label{eq:UGS-def}
\norm{x(t,x_0,u)} \le \sigma(\norm{x_0}) + \gamma(\norm{u}_{[0,t],2}) 
\qquad (t \in \R^+_0)
\end{align} 
is satisfied for every $(x_0,u) \in X \times L^2_{\mathrm{loc}}(\R^+_0,U)$. In particular, this estimate means that $x_0 := 0$ is an equilibrium point, and a globally stable one, of the system~\eqref{eq:wp-semilin-distr-contr/obs} or~\eqref{eq:wp-semilin-bdry-contr/obs} with input $u := 0$ and that this global stability property is impaired only slightly by external inputs $u$ of small magnitudes $\norm{u}_{[0,t],2}$.

\subsection{Semilinear systems without control or observation} 

We now collect some preliminaries on the solvability of semilinear evolution equations 
\begin{align} \label{eq:semilin-evol-eq-without-in/outputs}
\dot{x}(t) = A(t)x(t) + f(t,x(t))
\end{align}
without control inputs or observation outputs, which will be repeatedly used in the sequel. 
%
We recall that a family $A$ of operators $A(t): D(A(t)) \subset X \to X$ with $t \in \R^+_0$ is called \emph{locally Kato-stable}~\cite{Kato70}, \cite{Nickel00} iff $A(t)$ is a semigroup generator on $X$ for every $t\in \R^+_0$ and for every $t_0 \in (0,\infty)$ there exist constants $M_{t_0} \in [1,\infty)$ and $\omega_{t_0} \in \R$ such that
\begin{align} \label{eq:kato-stabilitaet,def}
\norm{ \e^{A(t_n) s_n} \dotsb \e^{A(t_1) s_1} } 
\le 
M_{t_0} \e^{\omega_{t_0}(s_1 + \dotsb + s_n)}
\end{align}
for all $s_1, \dots, s_n \in \R^+_0$ and all $t_1, \dots, t_n \in [0,t_0]$ satisfying $t_1 \le \dotsb \le t_n$ with arbitrary $n \in \N$.

\begin{cond} \label{ass:A(t)}
$A(t) = A_0(t)M(t)$ 
for $t \in \R^+_0$, where $A_0(t): D(A_0(t)) \subset X \to X$ are linear operators with time-independent domains $D(A_0(t)) = D_0$ and where $M(t) \in L(X)$ are bijective 
onto $X$ such that
\begin{itemize}
\item the family $M A_0$ consisting of the operators $M(t)A_0(t)$ is locally Kato-stable
\item $t \mapsto A_0(t)x$ is continuously differentiable for every $x \in D_0$ and $t \mapsto M(t)$ is twice strongly continuously differentiable.
\end{itemize}
\end{cond}

A simple sufficient condition for the above assumption to hold is provided by the following lemma. See Example~2.6 of~\cite{JaLa19}.

\begin{lm} \label{lm:vereinf-vor-A(t)}
Suppose that $X$ is a Hilbert space and $A(t) = A_0(t)M(t)$ for $t \in \R^+_0$, where 
\begin{itemize}
\item $A_0(t): D(A_0(t)) \subset X \to X$ are contraction semigroup generators on $X$ with time-independent domains $D(A_0(t)) = D_0$ and $t\mapsto A_0(t)x$ is continuously differentiable
\item $M(t) \in L(X)$ are symmetric 
and there exist constants $\ul{m}, \ol{m} \in (0,\infty)$ such that
\begin{align} \label{eq:M(t)-lower-and-upper-bounded}
\ul{m} \le M(t) \le \ol{m} \qquad (t\in \R^+_0),
\end{align}
and $t \mapsto M(t)$ is twice strongly continuously differentiable.
\end{itemize} 
Condition~\ref{ass:A(t)} is then satisfied.
\end{lm}

\begin{proof}
We have obviously only to prove the local Kato-stability of the family $MA_0$. In order to see that $M(t)A_0(t)$ is a semigroup generator on $X$ for every $t \in \R^+_0$, one can argue as in the proof of Lemma~7.2.3 of~\cite{JacobZwart} and in order to see that the semigroups generated by the operators $M(t)A_0(t)$ satsify estimates of the form~\eqref{eq:kato-stabilitaet,def}, one can argue as in the middle part of the proof of Proposition~2.3 from~\cite{ScWe10}.
\end{proof}

In the next result, we discuss the classical solvability of the linear problem
\begin{align} \label{eq:lin-evol-eq-without-in/outputs}
\dot{x}(t) = A(t)x(t) 
\end{align}
corresponding to~\eqref{eq:semilin-evol-eq-without-in/outputs}, that is, of~\eqref{eq:semilin-evol-eq-without-in/outputs} with $f(t,x) \equiv 0$. It is based on standard results~\cite{Kato56}, \cite{Kisynski63}, \cite{Kato70}  for 
non-autonomous linear evolution equations and slightly extends a solvability result from~\cite{ScWe10} (Proposition~2.8(a)), where the Hilbert space setting from Lemma~2.2 above is assumed. 
%
It is most conveniently formulated in terms of (solving) evolution systems. A \emph{(solving) evolution system for $A$ on the spaces $D(A(s))$}~\cite{EngelNagel}, \cite{diss} is, by definition, a family $T$ of bounded operators $T(t,s) \in L(X)$ for $(s,t) \in \Delta := \{(s,t) \in (\R^+_0)^2: s \le t\} $ such that 
\begin{itemize}
\item[(i)] for every $s \in \R^+_0$ and $x_s \in D(A(s))$, the map $[s,\infty) \ni t \mapsto T(t,s)x_s$ is a classical solution of~\eqref{eq:lin-evol-eq-without-in/outputs}
\item[(ii)] $T(t,s)T(t,r) = T(t,r)$ for all $(r,s), (s,t) \in \Delta$ and $\Delta \ni (s,t) \mapsto T(t,s)$ is strongly continuous. 
\end{itemize} 


\begin{lm} \label{lm:zeitentw-zu-A, ex-und-absch}
Suppose that $A(t) = A_0(t)M(t)$ are operators as in Condition~\ref{ass:A(t)}. Then there exists a unique evolution system $T$ for $A$ on the spaces $D(A(s))$ and for every $t_0 \in (0,\infty)$ there exist constants $M_{t_0} \in [1,\infty)$ and $\omega_{t_0} \in \R$ such that
\begin{align} \label{eq:zeitentw-zu-A,absch}
\norm{T(t,s)} \le M_{t_0} \e^{\omega_{t_0}(t-s)} \qquad ((s,t) \in \Delta_{[0,t_0]} := \Delta \cap [0,t_0]^2). 
\end{align}
\end{lm}

\begin{proof}
%
We have only to observe that the operator
\begin{align*}
A(t) = A_0(t)M(t) = M(t)^{-1} \big( M(t)A_0(t) \big) M(t)
\end{align*}
is similar to the operator $M(t)A_0(t)$ and then to combine -- in exactly the same way as in the proof of Corollary~2.1.10 of~\cite{diss} -- some standard results for non-autonomous linear evolution equations. We reproduce the arguments from~\cite{diss} here for the reader's convenience. 
Since by Condition~\ref{ass:A(t)} $$t \mapsto M(t)A_0(t)x + \dot{M}(t) M(t)^{-1}x$$ is continuously differentiable for every $x \in D_0$ and $M A_0 + \dot{M}M^{-1}$ is locally Kato-stable by Proposition~3.5 of~\cite{Kato70} with constants $\tilde{M}_{t_0}$, $\tilde{\omega}_{t_0}$ say, 
it follows from Theorem~6.1 of~\cite{Kato70} that there exists a unique evolution system $\tilde{T}_0$ for $A_0 + \dot{M}M^{-1}$ on the space $D_0$ and that
\begin{align} \label{eq:zeitentw-zu-A,transformierte-absch}
\| \tilde{T}_0(t,s) \| \le \tilde{M}_{t_0} \e^{\tilde{\omega}_{t_0}(t-s)} \qquad ((s,t) \in \Delta_{[0,t_0]}).
\end{align} 
Set now $T(t,s) := M(t)^{-1} \tilde{T}_0(t,s) M(s)$ for $(s,t) \in \Delta$. As is easily verified, $T$ is an evolution system for $A$ on the spaces $D(A(s))$ and the estimate~\eqref{eq:zeitentw-zu-A,transformierte-absch} yields the desired estimate~\eqref{eq:zeitentw-zu-A,absch} with appropriate constants $M_{t_0}$ and $\omega_{t_0} := \tilde{\omega}_{t_0}$.
\end{proof}

In the next result, we discuss the classical solvability of the full semilinear 
problem~\eqref{eq:semilin-evol-eq-without-in/outputs}. It is based on and extends  the solvability result from~\cite{Pr78}, where the linear parts $A(t)$ are assumed to have a time-independent domain.

\begin{lm} \label{lm:pruess-extended}
Suppose 
that $A(t) = A_0(t)M(t): D(A(t))\subset X \to X$ are operators as in Condition~\ref{ass:A(t)} on a reflexive space $X$ 
and that $f: \R^+_0 \times X \to X$ is Lipschitz on bounded subsets of $\R^+_0 \times X$.
Then for every $s \in \R^+_0$ and $x_s \in D(A(s))$, 
the system~\eqref{eq:semilin-evol-eq-without-in/outputs} has a unique maximal classical solution $x(\cdot,s,x_s) \in C^1([s,T_{s,x_s}),X)$ with initial state $x_s$ at initial time $s$. Additionally, this solution satisfies the integral equation
\begin{align} \label{eq:x(.,s,x-s)-mild solution}
x(t,s,x_s) = T(t,s)x_s + \int_s^t T(t,\tau)f\big(\tau,x(\tau,s,x_s)\big)\d \tau 
\qquad (t \in [s,T_{s,x_s}))
\end{align}
with the evolution system $T$ from Lemma~\ref{lm:zeitentw-zu-A, ex-und-absch}. And finally, this solution exists globally in time, that is, $T_{s,x_s} = \infty$, provided that it is bounded:
\begin{align}
\sup_{t\in [s,T_{s,x_s})} \norm{x(t,s,x_s)} < \infty.
\end{align}
\end{lm}

\begin{proof}
As a first step, we observe that the variable transformation $\xi(t) = M(t)x(t)$ induces a one-to-one correspondence between the maximal classical solutions of~\eqref{eq:semilin-evol-eq-without-in/outputs} and the maximal classical solutions of 
\begin{align} \label{eq:semilin-evol-eq-without-in/outputs,transformed}
\dot{\xi}(t) = M(t)A_0(t)\xi(t) + M(t) f\big(t,M(t)^{-1}\xi(t)\big) + \dot{M}(t)M(t)^{-1}\xi(t).
\end{align}
Indeed, it is elementary to verify that for a (maximal) classical solution $x:J \to X$ of~\eqref{eq:semilin-evol-eq-without-in/outputs} the function $\xi:J \to X$ defined by $\xi(t):= M(t)x(t)$ is a (maximal) classical solution of~\eqref{eq:semilin-evol-eq-without-in/outputs,transformed} and that, conversely, for a (maximal) classical solution $\xi: J \to X$ of~\eqref{eq:semilin-evol-eq-without-in/outputs,transformed} the function $x:J \to X$ is a (maximal) classical solution of~\eqref{eq:semilin-evol-eq-without-in/outputs}. 
\smallskip

As a second step, we show that for every $s \in \R^+_0$ and $x_s \in D(A(s))$, 
the system~\eqref{eq:semilin-evol-eq-without-in/outputs} has a unique maximal classical solution $x(\cdot,s,x_s) \in C^1([s,T_{s,x_s}),X)$ with initial state $x_s$ at initial time $s$.
So let $s \in \R^+_0$ and $x_s \in D(A(s))$. We want to apply the solvability result (Theorem~1) from~\cite{Pr80} to the transformed equation~\eqref{eq:semilin-evol-eq-without-in/outputs,transformed}. It is clear by Condition~\ref{ass:A(t)} that the assumptions of Theorem~1 of~\cite{Pr80} 
are satisfied. 
It also follows, by the very same arguments as in the autonomous case~\cite{Pazy}, that~\eqref{eq:semilin-evol-eq-without-in/outputs,transformed} has a unique maximal mild solution $\xi(\cdot,s,\xi_s) \in C(J_{s,x_s},X)$ with initial state $\xi_s := M(s)x_s$ at initial time $s$ and that the maximal existence interval $J_{s,x_s}$ is half-open: $J_{s,x_s} = [s,T_{s,x_s})$. Since now 
\begin{align*}
\xi_s = M(s)x_s \in M(s)D(A(s)) = D(A_0(s)) = D_0
\end{align*} 
and $X$ is reflexive, Theorem~1 of~\cite{Pr80} 
implies that $\xi(\cdot,s,\xi_s)$ is also a classical solution of~\eqref{eq:semilin-evol-eq-without-in/outputs,transformed}. Since, moreover, classical solutions 
are well-known to be also mild solutions of~\eqref{eq:semilin-evol-eq-without-in/outputs,transformed} and since every mild solution of~\eqref{eq:semilin-evol-eq-without-in/outputs,transformed} with initial state $\xi_s$ is a restriction of the maximal mild solution $\xi(\cdot,s,\xi_s)$, 
$\xi(\cdot,s,\xi_s)$ even is a unique maximal classical solution. So, by the first step, $$x(\cdot,s,x_s) := M(\cdot) \xi(\cdot,s,\xi_s) \in C^1([s,T_{s,x_s}),X)$$ is a unique maximal classical solution with initial state $M(s)\xi_s = x_s$, as desired. 
\smallskip

As a third step, we prove the integral equation~\eqref{eq:x(.,s,x-s)-mild solution} for $s \in \R^+_0$ and $x(s) \in D(A(s))$, which says that the classical solutions $x(\cdot,s,x_s)$ is also a mild solution of~\eqref{eq:semilin-evol-eq-without-in/outputs}.
So let $s \in \R^+_0$ and $x_s \in D(A(s))$ and let $t \in [s,T_{s,x_s})$ be fixed. It follows by 
the right differentiabily property (Lemma~2.1.5 of~\cite{diss}) of the evolution system $T$ for $A$ on the spaces $D(A(s))$ 
that
\begin{align} \label{eq:pruess,step-3,1}
[s,t] \ni \tau \mapsto T(t,\tau)x(\tau,s,x_s)
\end{align}
is continuous and right differentiable with right derivative
\begin{align} \label{eq:pruess,step-3,2}
[s,t] \ni \tau \mapsto T(t,\tau)f(\tau,x(\tau,s,x_s)).
\end{align}
Since this right derivative is continuous, it further follows by Corollary~2.1.2 of~\cite{Pazy} that~\eqref{eq:pruess,step-3,1} is continuously differentiable with derivative~\eqref{eq:pruess,step-3,2}. And therefore we obtain~\eqref{eq:x(.,s,x-s)-mild solution} by the fundamental theorem of calculus. 
\smallskip

As a fourth step, we show that 
if for some $s\in \R^+_0$ and $x_s \in D(A(s))$ the maximal classical solution $x(\cdot,s,x_s)$ does not exist globally in time, that is, if $T_{s,x_s} < \infty$, then it must be unbounded:
\begin{align} \label{eq:pruess,step-4}
\sup_{t\in [s,T_{s,x_s})} \norm{x(t,s,x_s)} = \infty.
\end{align}
So assume that $T_{s,x_s} < \infty$ for some $s\in \R^+_0$ and $x_s \in D(A(s))$. We want to apply the global solvability result (Theorem~6) from~\cite{Pr78} to the transformed equation~\eqref{eq:semilin-evol-eq-without-in/outputs,transformed}. It is clear by Condition~\ref{ass:A(t)} that the linear and nonlinear part of~\eqref{eq:semilin-evol-eq-without-in/outputs,transformed} satisfy the assumptions of Theorem~6 of~\cite{Pr78}. It also follows, by the first step and the proof of the second step, that $\xi: [s,T_{s,x_s}) \to X$ defined by $\xi(t) := M(t)x(t,s,x_s)$ is a maximal mild solution of~\eqref{eq:semilin-evol-eq-without-in/outputs,transformed}. Since now this maximal mild solution does not exist on $[s,\infty)$ but only on $[s,T_{s,x_s})$, Theorem~6 of~\cite{Pr78} implies that 
\begin{align*}
\sup_{t\in [s,T_{s,x_s})} \norm{\xi(t)} = \infty.
\end{align*}
Since, moreover, $t \mapsto M(t)$ is locally bounded and $[0,T_{s,x_s}]$ is compact, we conclude that
\begin{align*}
\sup_{t\in [s,T_{s,x_s})} \norm{x(t,s,x_s)} \ge \Big( \sup_{t\in [s,T_{s,x_s}]} \norm{M(t)} \Big)^{-1} \sup_{t\in [s,T_{s,x_s})} \norm{\xi(t)} = \infty,
\end{align*}
that is, \eqref{eq:pruess,step-4} is satisfied, 
as desired. 
\end{proof}

As a last preliminary, we record two simple facts for later reference. We give the elementary proofs for the sake of completeness.

\begin{lm} \label{lm:density-and-choice-of-lipschitz-constants}
\begin{itemize}
\item[(i)] If $f: \R^+_0 \times X \to X$ is Lipschitz on bounded subsets of $\R^+_0 \times X$, then 
one can choose Lipschitz constants $L_{\rho}$ of $f|_{[0,\rho] \times \ol{B}_{\rho}(0)}$ for 
$\rho \in \R^+_0$ such that $\rho \mapsto L_{\rho}$ is continuous and monotonically increasing. 
\item[(ii)] $C_{\mathrm{c}}^2(\R^+_0,U)$ is dense in $L^2_{\mathrm{loc}}(\R^+_0,U)$,  where $U$ is an arbitrary Banach space.
\end{itemize}
\end{lm}

\begin{proof}
(i) We have only to apply the elementary and well-known fact 
that any monotonically increasing function $l: \R^+_0 \to \R^+_0$ 
can be majorized by a continuous monotonically increasing function to the particular function
\begin{align*}
\rho \mapsto L_{\rho}^0 := \min \Big\{ L \in \R^+_0: L \text{ is a Lipschitz constant of } f\big|_{[0,\rho] \times \ol{B}_{\rho}(0)} \Big\} < \infty.
\end{align*}
See Lemma~2.5 of~\cite{ClLeSt98}, for instance. 
(ii) We have to show that for a given $u \in L^2_{\mathrm{loc}}(\R^+_0,U)$ there exists a sequence $(u_n)$ in $C_{\mathrm{c}}^2(\R^+_0,U)$ such that
\begin{align} \label{eq:density,bew}
\norm{u_n - u}_{[0,t_0],2} \longrightarrow 0 \qquad (n \to \infty)
\end{align}
for every $t_0 \in (0,\infty)$. So let $u \in L^2_{\mathrm{loc}}(\R^+_0,U)$. Since $u|_{[0,n]} \in L^2([0,n],U)$ and $C_{\mathrm{c}}^2((0,n),U)$ is dense in $L^2([0,n],U)$, for every $n \in \N$ there exists a function $u_n \in C_{\mathrm{c}}^2(\R^+,U)$ with
\begin{align*}
\supp u_n \subset (0,n) \qquad \text{and} \qquad \norm{u_n-u}_{[0,n],2} \le 1/n.
\end{align*} 
So, for every given $t_0 \in (0,\infty)$, we have $\norm{u_n - u}_{[0,t_0],2} \le \norm{u_n-u}_{[0,n],2} \le 1/n$ provided that $n \ge t_0$ and therefore~\eqref{eq:density,bew} follows. 
\end{proof}

\section{Semilinear systems with bounded control and observation operators} \label{sect:wp-bd-io-op}

In this section, we establish the well-posedness and uniform global stability of semilinear systems~\eqref{eq:wp-semilin-distr-contr/obs} with bounded control and observation operators~\eqref{eq:bd-io-op}.

\subsection{Classical solutions and outputs}

We begin by establishing the existence and uniqueness of classical solutions  for sufficiently regular initial states and inputs. In order to do so, we make the following assumptions. 

\begin{cond}  \label{ass:1, distr}
$X$ is a reflexive Banach space and
\begin{itemize}
\item[(i)] 
$A(t) := \mathcal{A}(t)$ are operators as in Condition~\ref{ass:A(t)}

\item[(ii)] $\mathcal{B}(t): U \to X$ and $\mathcal{C}(t): X \to Y$ are bounded linear operators and, moreover, $t \mapsto \mathcal{B}(t)$ is locally Lipschitz

\item[(iii)] $f$ is Lipschitz on bounded subsets of $\R^+_0 \times X$.
\end{itemize}
\end{cond}

\begin{lm} \label{lm:loc-ex, distr}
If Condition~\ref{ass:1, distr} is satisfied, then for every $s \in \R^+_0$ and every classical datum $(x_s,u) \in \mathcal{D}_s$ 
with 
\begin{align} \label{eq:D_s-distr-def}
\mathcal{D}_s := D(\mathcal{A}(s)) \times C_c^2(\R^+_0,U),
\end{align}
the system~\eqref{eq:wp-semilin-distr-contr/obs} has a unique maximal classical solution $x(\cdot,s,x_s,u) \in C^1([s,T_{s,x_s,u}),X)$. 
\end{lm}

\begin{proof}
In view of Condition~\ref{ass:1, distr} and Lemma~\ref{lm:pruess-extended}, the evolution equation
\begin{align} \label{eq:original-system, bd}
\dot{x}(t) = \mathcal{A}(t)x(t) + f(t,x(t)) + \mathcal{B}(t)u(t)
\end{align}
has a unique maximal classical solution for every initial state $x_s \in D(\mathcal{A}(s))$ and every input $u \in C_c^2(\R^+_0,U)$, as desired.
\end{proof}

In the entire Section~\ref{sect:wp-bd-io-op}, the symbols $\mathcal{D}_s$ and $x(\cdot,s,x_s,u)$ will have the meaning from the above lemma. Also, we will write $x(\cdot,x_0,u) := x(\cdot,0,x_0,u)$ for brevity.  
In order to obtain globality of the maximal classical solutions from above, we make the following additional assumptions. 

\begin{cond} \label{ass:2, distr}
\begin{itemize}
\item[(i)] System~\eqref{eq:wp-semilin-distr-contr/obs} is scattering-passive w.r.t.~a continuously differentiable storage function $V$, that is, $V \in C^1(\R^+_0\times X,\R^+_0)$ and for some $\alpha, \beta > 0$
\begin{gather}
\dot{V}(t,x(t,s,x_s,u)) \le \alpha \norm{u(t)}_U^2 - \beta \norm{y(t,s,x_s,u)}_Y^2
\qquad (t \in [s,T_{s,x_s,u}))
\label{eq:scattering-passive, distr}
\\
(y(t,s,x_s,u) := \mathcal{C}(t)x(t,s,x_s,u)) \notag
\end{gather}
for every $s \in \R^+_0$ 
and every $(x_s,u) \in \mathcal{D}_s$
\item[(ii)] $V(t,\cdot)$ is equivalent to the norm $\norm{\cdot}$ of $X$ uniformly w.r.t.~$t$, that is, for some $\ul{\psi}, \ol{\psi} \in \mathcal{K}_{\infty}$
\begin{align} \label{eq:V-equiv-to-norm, distr}
\ul{\psi}(\norm{x}) \le V(t,x) \le \ol{\psi}(\norm{x}) 
\qquad ((t,x) \in \R^+_0 \times X).
\end{align} 
\end{itemize}
\end{cond}

\begin{lm} \label{lm:glob-ex-UGS,distr}
If Condition~\ref{ass:1, distr} and~\ref{ass:2, distr} are satisfied, then the maximal classical solution $x(\cdot,s,x_s,u)$ 
exists globally in time for every $s \in \R^+_0$ and $(x_s,u) \in \mathcal{D}_s$, that is, $T_{s,x_s,u} = \infty$. Additionally, there exist $\sigma, \gamma \in \mathcal{K}$ such that
\begin{align} \label{eq:UGS, distr}
\norm{x(t,s,x_s,u)} \le \sigma(\norm{x_s}) + \gamma(\norm{u}_{[s,t],2})
\qquad (t \in [s,\infty))
\end{align}
for every $s \in \R^+_0$ and $(x_s,u) \in \mathcal{D}_s$.
\end{lm}

\begin{proof}
We first show that there exist $\sigma, \gamma \in \mathcal{K}$ such that the estimate~\eqref{eq:UGS, distr} is satisfied at least for all $t \in [s,T_{s,x_s,u})$. 
So, let $s \in \R^+_0$ and $(x_s,u) \in \mathcal{D}_s$. Integrating~\eqref{eq:scattering-passive, distr}, 
we obtain
\begin{align*}
V(t,x(t,s,x_s,u)) = V(s,x_s) + \int_s^t \dot{V}(\tau,x(\tau,s,x_s,u)) \d\tau \le V(s,x_s) + \alpha \int_s^t \norm{u(\tau)}_U^2 \d\tau
\end{align*}
and therefore by~\eqref{eq:V-equiv-to-norm, distr}
\begin{align*}
\ul{\psi}(\norm{x(t,s,x_s,u)}) \le \ol{\psi}(\norm{x_s}) + \alpha \norm{u}_{[s,t],2}^2
\end{align*}
for every $t \in [s,T_{s,x_s,u})$. Consequently, 
\begin{align*}
\norm{x(t,s,x_s,u)} \le \ul{\psi}^{-1}\big( \ol{\psi}(\norm{x_s}) + \alpha \norm{u}_{[s,t],2}^2 \big) \le \ul{\psi}^{-1}\big( 2\ol{\psi}(\norm{x_s}) \big) + \ul{\psi}^{-1}\big( 2\alpha \norm{u}_{[s,t],2}^2 \big) 
\end{align*}
for every $t \in [s,T_{s,x_s,u})$, that is, for $\sigma, \gamma$ defined by $\sigma(r) := \ul{\psi}^{-1}(2\ol{\psi}(r))$ and $\gamma(r) := \ul{\psi}^{-1}(2\alpha r^2)$ the desired estimate follows and $\sigma, \gamma \in \mathcal{K}$. 
%
With this estimate and the compactness of $\supp u$, in turn, we obtain
\begin{align*}
\sup_{t \in [s,T_{s,x_s,u})} \norm{x(t,s,x_s,u)} \le \sigma(\norm{x_s}) + \gamma(\norm{u}_{[s,\infty),2}) < \infty
\end{align*}
for every $s \in \R^+_0$ and $(x_s,u) \in \mathcal{D}_s$. 
So, as $x(\cdot,s,x_s,u)$ is a maximal classical solution of~\eqref{eq:original-system, bd}, it is even global by virtue of Lemma~\ref{lm:pruess-extended}. In other words, $T_{s,x_s,u} = \infty$, as desired. 
\end{proof}

\subsection{Well-posedness: generalized solutions and outputs}

With the above preliminaries on classical solutions, we can now move on to establish also the unique existence of generalized solutions and outputs. In order to do so, we exploit the following integral equation for classical solutions of~\eqref{eq:wp-semilin-bdry-contr/obs}:
\begin{align} \label{eq:int-eq-distr}
x(t,x_0,u) = T(t,0)x_0 + \int_0^t T(t,s)f(s,x(s,x_0,u))\d s + \Phi_t(u)
\end{align}
for all $t \in \R^+_0$ and $(x_0,u) \in \mathcal{D}_0$, where  
\begin{gather}
\Phi_t(u) := \int_0^t T(t,s) \mathcal{B}(s)u(s) \d s 
\label{eq:def-Phi-t,distr}
\end{gather}
and where $T$ is the evolution system for $A$ on the spaces $D(A(s))$ (Lemma~\ref{lm:zeitentw-zu-A, ex-und-absch}). (Invoke Lemma~\ref{lm:pruess-extended} to obtain the integral equation~\eqref{eq:int-eq-distr}.)

\begin{lm} \label{lm:dichtheit-und-Phi-t-stetig-fb,distr}
If Condition~\ref{ass:1, distr} is satisfied, then
\begin{itemize}
\item[(i)] $\mathcal{D}_0$ is dense in $X \times L^2_{\mathrm{loc}}(\R^+_0,U)$
\item[(ii)] $\Phi_t: C^2([0,t],U) \to X$ defined by~\eqref{eq:def-Phi-t,distr} for every $t \in (0,\infty)$ uniquely extends to a bounded linear operator $\ol{\Phi}_t: L^2([0,t],U) \to X$  and 
\begin{align} \label{eq:absch-Phi-t,distr}
C_{t_0} := \sup_{t\in [0,t_0]} \norm{\ol{\Phi}_t} < \infty
\end{align}
for every $t_0 \in (0,\infty)$.
\end{itemize}
\end{lm}

\begin{proof}
Assertion~(i) is a 
consequence of the density of $C_{\mathrm{c}}^2(\R^+_0,U)$ in $L^2_{\mathrm{loc}}(\R^+_0,U)$ (Lem-ma~\ref{lm:density-and-choice-of-lipschitz-constants}) and the density of $D(A(0)) = M(0)^{-1}D(A_0(0))$ in $X$ (Condition~\ref{ass:1, distr}(i)). 
\smallskip

Assertion~(ii) immediately follows from the definition of $\Phi_t$. 
Indeed, let $t_0 \in (0,\infty)$ be fixed and let $t \in [0,t_0]$ and
\begin{align} \label{eq:Phi-t-bd,distr}
u_t \in C^2([0,t],U) \qquad \text{with} \qquad \norm{u_t}_{[0,t],2} \le 1. 
\end{align}
We can then conclude from~\eqref{eq:def-Phi-t,distr} with the help of~\eqref{eq:zeitentw-zu-A,absch} and Condition~\ref{ass:1, distr}(ii) that
\begin{align} \label{eq:absch-Phi-t,distr,beweis}
\norm{\Phi_t(u_t)} \le M_{t_0}\e^{\omega_{t_0} t_0} b_{t_0} \int_0^t \norm{u(s)}_U \d s
\le M_{t_0}\e^{\omega_{t_0} t_0} b_{t_0} t_0^{1/2} \norm{u_t}_{[0,t],2},
\end{align}
where $b_{t_0} := \sup_{s \in [0,t_0]} \norm{\mathcal{B}(s)}_{U,X}$. And from~\eqref{eq:absch-Phi-t,distr,beweis} and~\eqref{eq:Phi-t-bd,distr}, in turn, the assertion~(ii) is clear.
\end{proof}

With the assumptions and preparations made so far, we can already prove the existence of unique generalized solutions and their continuous dependence on the data. In order to get the same things also for  generalized outputs, we impose the following final assumption. 
(It should be noticed that especially the measurability part~(ii) of that assumption is fairly weak. Indeed, by the boundedness of the output operators, 
this measurability condition is already satisfied if only $t \mapsto \mathcal{C}(t)$ is strongly measurable.)

\begin{cond}  \label{ass:3, distr}
\begin{itemize}
\item[(i)] $\partial V: \R^+_0 \times X \to L(\R\times X,\R)$, the derivative of $V$, is bounded on bounded subsets of $\R^+_0 \times X$
\item[(ii)] $y(\cdot,s,x_s,u) := \mathcal{C}(\cdot)x(\cdot,s,x_s,u)$ is measurable for every $s \in \R^+_0$ and $(x_s,u) \in \mathcal{D}_s$.
\end{itemize}
\end{cond}


\begin{thm} \label{thm:wp, distr}
If Condition~\ref{ass:1, distr}, \ref{ass:2, distr} and~\ref{ass:3, distr} are satisfied and if $f(t,0) =0$ for every $t \in \R^+_0$, then the system~\eqref{eq:wp-semilin-distr-contr/obs} is well-posed
and uniformly globally stable. 
\end{thm}

\begin{proof}
(i) We first show that for every $t_0 \in (0,\infty)$ and every $(x_{01},u_1), (x_{02},u_2) \in \mathcal{D}_0$ one has the following fundamental estimate:
\begin{align} \label{eq:bew,general-solution,distr,zentrale-absch}
\|x(\cdot,x_{01},u_1) &- x(\cdot,x_{02},u_2)\|_{[0,t_0],\infty} 
\le 
\Big( M_{t_0}\e^{\omega_{t_0} t_0} \norm{x_{01}-x_{02}} + C_{t_0} \norm{u_1-u_2}_{[0,t_0],2} \Big) \cdot \notag \\
&\quad \cdot \exp\big( M_{t_0}\e^{\omega_{t_0} t_0} L_{t_0+\rho_{t_0}(x_{01},u_1,x_{02},u_2)} t_0 \big), 
\end{align}
where $M_{t_0}$, $\omega_{t_0}$ are as in Lemma~\ref{lm:zeitentw-zu-A, ex-und-absch}, 
$C_{t_0}$ is as in Lemma~\ref{lm:dichtheit-und-Phi-t-stetig-fb,distr}, 
$L_{\rho}$ are Lipschitz constants chosen as in Lemma~\ref{lm:density-and-choice-of-lipschitz-constants} and 
\begin{align} \label{eq:def-rho-vier-argumente,distr}
\rho_{t_0}(x_{01},u_1,x_{02},u_2) := \sigma(\norm{x_{01}}) + \gamma(\norm{u_1}_{[0,t_0],2}) + \sigma(\norm{x_{02}}) + \gamma(\norm{u_2}_{[0,t_0],2})
\end{align} 
with $\sigma, \gamma$ as in Lemma~\ref{lm:glob-ex-UGS,distr}. 
So let $t_0 \in (0,\infty)$ and $(x_{01},u_1)$, $(x_{02},u_2) \in \mathcal{D}_0$ and write $x_i := x(\cdot,x_{0i},u_i)$. 
It then follows from~\eqref{eq:int-eq-distr} with the help of~\eqref{eq:zeitentw-zu-A,absch}, \eqref{eq:absch-Phi-t,distr}, and the Lipschitz continuity of $f$ on bounded subsets (Condition~\ref{ass:1, distr}(iii)) combined with~\eqref{eq:UGS, distr} that
\begin{align}
\norm{x_1(t)-x_2(t)} &\le M_{t_0}\e^{\omega_{t_0} t_0} \norm{x_{01}-x_{02}} + C_{t_0} \norm{u_1-u_2}_{[0,t_0],2} \notag \\
&\quad \qquad + M_{t_0}\e^{\omega_{t_0} t_0} \int_0^t L_{t_0+\rho_{t_0}(x_{01},u_1,x_{02},u_2)} \norm{x_1(s)-x_2(s)} \d s
\end{align}
for all $t \in [0,t_0]$. And therefore the desired estimate~\eqref{eq:bew,general-solution,distr,zentrale-absch} follows by virtue of Gr\"onwall's lemma. 
\smallskip

Combining this estimate~\eqref{eq:bew,general-solution,distr,zentrale-absch} now with the density of $\mathcal{D}_0$ in $X\times L^2_{\mathrm{loc}}(\R^+_0,U)$ (Lem-ma~\ref{lm:dichtheit-und-Phi-t-stetig-fb,distr}), we immediately 
see that for every $(x_0,u) \in X\times L^2_{\mathrm{loc}}(\R^+_0,U)$ there exists a unique generalized solution 
\begin{align} \label{eq:generalized-sol,distr}
x(\cdot,x_0,u) \in C(\R^+_0,X)
\end{align}
to~\eqref{eq:wp-semilin-distr-contr/obs}. (In order to see the asserted uniqueness, notice that the approximating data $(x_{0n},u_n) \in X \times C_c^2(\R^+_0,U)$ from the definition of generalized solutions and outputs actually belong to the classical data set $\mathcal{D}_0$ from Lemma~\ref{lm:loc-ex, distr} -- because of the 
classical solvability requirement made in the definition before~\eqref{eq:wp-def} for~\eqref{eq:wp-semilin-distr-contr/obs} with data $(x_{0n},u_n)$.) 
Since the right-hand side of the estimate~\eqref{eq:UGS, distr} with $s=0$ depends continuously on $(x_0,u)$, this estimate extends from $\mathcal{D}_0$ to arbitrary $(x_0,u) \in X \times L^2_{\mathrm{loc}}(\R^+_0,U)$. Consequently, \eqref{eq:wp-semilin-distr-contr/obs} is uniformly globally stable.
Since, moreover, the right-hand side of~\eqref{eq:bew,general-solution,distr,zentrale-absch} depends continuously on $(x_{01},u_1), (x_{02},u_2)$, 
this estimate extends from $\mathcal{D}_0$ to arbitrary $(x_{01},u_1), (x_{02},u_2) \in X\times L^2_{\mathrm{loc}}(\R^+_0,U)$. And this extended estimate, 
in turn, yields the continuity of the generalized solution map $(x_0,u) \mapsto x(\cdot,x_0,u) \in C(\R^+_0,X)$. 
\smallskip

(ii) We first show that for every $t_0 \in (0,\infty)$ and every $(x_{01},u_1), (x_{02},u_2) \in \mathcal{D}_0$ one has the following fundamental estimate:
\begin{align} \label{eq:bew,general-output,distr,zentrale-absch}
&\beta \norm{y(\cdot,x_{01},u_1)-y(\cdot,x_{02},u_2)}_{[0,t_0],2}^2 
\le 
\ol{\psi}(\norm{x_{01}-x_{02}}) + \alpha \norm{u_1-u_2}_{[0,t_0],2}^2 \notag \\
&\qquad \qquad \qquad \, + 2 M_{t_0+\rho_{t_0}(x_{01},u_1,x_{02},u_2)} \norm{x(\cdot,x_{01},u_1)-x(\cdot,x_{02},u_2)}_{[0,t_0],\infty} t_0,  
\end{align}
where $\rho_{t_0}(x_{01},u_1,x_{02},u_2)$ is defined as in~\eqref{eq:def-rho-vier-argumente,distr} 
and where 
$M_{\rho} := K_{\rho} L_{\rho}$ with 
\begin{align*}
K_{\rho} \ge  K_{\rho}^0 := \sup \{ \norm{\partial V(t,x)}: t + \norm{x} \le \rho \}
\end{align*}
chosen such that $\rho \mapsto K_{\rho}$ is continuous and monotonically increasing (see the proof of Lemma~\ref{lm:density-and-choice-of-lipschitz-constants}) and with Lipschitz constants $L_{\rho}$ chosen as in Lemma~\ref{lm:density-and-choice-of-lipschitz-constants}. 
%
So let $t_0 \in (0,\infty)$ and $(x_{01},u_1)$, $(x_{02},u_2) \in \mathcal{D}_0$ and write $x_i := x(\cdot,x_{0i},u_i)$ and $y_i := y(\cdot,x_{0i},u_i) = \mathcal{C}(\cdot)x(\cdot,x_{0i},u_i)$ as well as $x_{12} := x_1-x_2$ and $u_{12} := u_1-u_2$.  
It then follows by the differential equation~\eqref{eq:wp-semilin-distr-contr/obs}  that
\begin{align} \label{eq:bew,general-output,distr,1}
\frac{\d}{\d s} V(s,x_{12}(s)) &= \partial_s V(s,x_{12}(s)) + \partial_x V(s,x_{12}(s))\big( \mathcal{A}(s)x_{12}(s) + f(s,x_{12}(s)) + \mathcal{B}(s)u_{12}(s) \big) \notag \\
&\quad + \partial_x V(s,x_{12}(s))\big( f(s,x_1(s))-f(s,x_2(s)) - f(s,x_{12}(s)) \big) 
\end{align} 
for all $s\in \R^+_0$. Since by the classical solution property of $x_i$ for~\eqref{eq:wp-semilin-distr-contr/obs}
\begin{align*}
x_{12}(s) = x_1(s) - x_2(s) \in D(A(s)) 
\qquad \text{and} \qquad
u_{12} \in C_{\mathrm{c}}^2(\R^+_0,U),
\end{align*}
we have $(x_{12s},u_{12}) := (x_{12}(s),u_{12}) \in \mathcal{D}_s$ and thus it follows by~\eqref{eq:scattering-passive, distr} that the first part of the right-hand side of~\eqref{eq:bew,general-output,distr,1} can be estimated as follows: 
\begin{align} \label{eq:bew,general-output,distr,2}
\partial_s V(s,x_{12}(s)) &+ \partial_x V(s,x_{12}(s))\big( \mathcal{A}(s)x_{12}(s) + f(s,x_{12}(s)) + \mathcal{B}(s)u_{12}(s) \big) \notag \\
&= \frac{\d}{\d t} V\big(t,x(t,s,x_{12s},u_{12})\big) \big|_{t=s}
\le \alpha \norm{u_{12}(s)}_U^2 - \beta \norm{\mathcal{C}(s)x_{12s}}_Y^2 \notag \\
&= \alpha \norm{u_1(s)-u_2(s)}_U^2 - \beta \norm{y_1(s)-y_2(s)}_Y^2 
\qquad (s \in \R^+_0). 
\end{align}
Since, moreover, $\partial V$ and $f$ are bounded or Lipschitz, respectively, on bounded subsets of $\R^+_0 \times X$ (Condition~\ref{ass:3, distr}(i) and~\ref{ass:1, distr}(iii)!) and $f(s,0)=0$, it further follows by~\eqref{eq:UGS, distr} that the second part of the right-hand side of~\eqref{eq:bew,general-output,distr,1} can be estimated as follows: 
\begin{align} \label{eq:bew,general-output,distr,3}
\partial_x V(s,x_{12}(s)) & \big( f(s,x_1(s)) - f(s,x_2(s)) - f(s,x_{12}(s)) \big) \notag \\
&\qquad \le 2 M_{t_0 + \rho_{t_0}(x_{01},u_1,x_{02},u_2)} \norm{x_1(s)-x_2(s)}
\qquad (s\in\R^+_0). 
\end{align} 
Inserting now~\eqref{eq:bew,general-output,distr,2} and~\eqref{eq:bew,general-output,distr,3} into~\eqref{eq:bew,general-output,distr,1} and integrating the resulting estimate (Condition~\ref{ass:3, distr}(ii)!), 
we finally obtain 
the desired estimate~\eqref{eq:bew,general-output,distr,zentrale-absch}. 
\smallskip

Combining this estimate~\eqref{eq:bew,general-output,distr,zentrale-absch} now with the density of $\mathcal{D}_0$ in $X\times L^2_{\mathrm{loc}}(\R^+_0,U)$ (Lem-ma~\ref{lm:dichtheit-und-Phi-t-stetig-fb,distr}) and the continuity of 
$(x_0,u) \mapsto x(\cdot,x_0,u) \in C(\R^+_0,X)$ established in part~(i) above, we immediately 
see that for every $(x_0,u) \in X\times L^2_{\mathrm{loc}}(\R^+_0,U)$ there exists a unique generalized output $$y(\cdot,x_0,u) \in L^2_{\mathrm{loc}}(\R^+_0,Y)$$ to~\eqref{eq:wp-semilin-distr-contr/obs}.  
Since, moreover, the right-hand side of~\eqref{eq:bew,general-output,distr,zentrale-absch} depends continuously on $(x_{01},u_1)$, $(x_{02},u_2)$, 
this estimate extends from $\mathcal{D}_0$ to arbitrary $(x_{01},u_1), (x_{02},u_2) \in X\times L^2_{\mathrm{loc}}(\R^+_0,U)$. And this extended estimate, 
in turn, yields the continuity of the generalized output map $(x_0,u) \mapsto y(\cdot,x_0,u) \in L^2_{\mathrm{loc}}(\R^+_0,Y)$. 
\end{proof}

\section{Semilinear systems with unbounded control and observation operators} \label{sect:wp-unbd-io-op}

In this section, we establish the well-posedness and uniform global stability of semilinear systems~\eqref{eq:wp-semilin-bdry-contr/obs} with unbounded control and observation operators~\eqref{eq:unbd-io-op}. 

\subsection{Classical solutions and outputs}

We begin by establishing the existence and uniqueness of classical solutions  for sufficiently regular initial states and inputs. In order to do so, we make the following assumptions. 

\begin{cond}  \label{ass:1, bdry}
$X$ is a reflexive Banach space and
\begin{itemize}
\item[(i)] 
$D(\mathcal{A}(t)) = D(\mathcal{B}(t))$ and $A(t) := \mathcal{A}(t)|_{\ker\mathcal{B}(t)}$ are operators as in Condition~\ref{ass:A(t)}

\item[(ii)] $\mathcal{B}(t): D(\mathcal{B}(t)) \subset X \to U$ and $\mathcal{C}(t): D(\mathcal{C}(t)) \subset X \to Y$ are linear operators with $D(\mathcal{C}(t)) \supset D(\mathcal{A}(t))$ and, moreover, $\mathcal{B}(t)$ has a bounded linear right-inverse $\mathcal{R}(t) \in L(U,X)$ for every $t \in \R^+_0$, that is,  
\begin{gather*}
\mathcal{R}(t)U \subset D(\mathcal{B}(t)) 
\qquad \text{and} \qquad 
\mathcal{B}(t) \mathcal{R}(t)u = u \qquad (t\in \R^+_0, u \in U),
\end{gather*}
such that $\mathcal{A}(t)\mathcal{R}(t) \in L(U,X)$ and such that $t \mapsto \mathcal{R}(t), \dot{\mathcal{R}}(t), \mathcal{A}(t)\mathcal{R}(t)$ are locally Lipschitz 

\item[(iii)] $f$ is Lipschitz on bounded subsets of $\R^+_0 \times X$.
\end{itemize}
\end{cond}

%
%

\begin{lm}
If Condition~\ref{ass:1, bdry} is satisfied, then for every $s \in \R^+_0$ and every classical datum $(x_s,u) \in \mathcal{D}_s$ with 
\begin{align}  \label{eq:D_s-bdry-def}
\mathcal{D}_s :=  \{ (x_s,u) \in D(\mathcal{A}(s)) \times C_c^2(\R^+_0,U): \mathcal{B}(s)x_s = u(s) \},
\end{align}
the system~\eqref{eq:wp-semilin-bdry-contr/obs} has a unique maximal classical solution $x(\cdot,s,x_s,u) \in C^1([s,T_{s,x_s,u}),X)$. 
\end{lm}

\begin{proof}
In contrast to the case with bounded in- and output operators, we cannot directly apply Lemma~\ref{lm:pruess-extended} anymore -- just because the evolution equation
\begin{align} \label{eq:original-system}
\dot{x}(t) = \mathcal{A}(t)x(t) + f(t,x(t)), 
\qquad
u(t) = \mathcal{B}(t)x(t)
\end{align}
we are interested in here is not of the form considered in that lemma. With the transformation
\begin{align} \label{eq:Fattorini-trf}
\xi(t) = x(t)-\mathcal{R}(t)u(t) 
\end{align}
however (which is well-known for linear systems~\cite{Fa68}, \cite{JacobZwart}), we can bring~\eqref{eq:original-system} to the desired form. 
Indeed, by virtue of Condition~\ref{ass:1, bdry}(i) and~(ii), we have
\begin{align} \label{eq:alternative-def-of-classical-dataset,bdry}
\mathcal{D}_t = \{ (x_t,u) \in X \times C_c^2(\R^+_0,U): x_t - \mathcal{R}(t)u(t) \in D(A(t)) \}
\end{align}
for every $t \in \R^+_0$ (note that $D(A(t)) = \ker \mathcal{B}(t)$ by assumption). And therefore, the transformation~\eqref{eq:Fattorini-trf} with $u \in C^1(\R^+_0,U)$ induces a one-to-one correspondence between the (maximal) classical solutions of~\eqref{eq:original-system} and the (maximal) classical solutions of 
\begin{align}  \label{eq:transformed-system}
\dot{\xi}(t) = A(t)\xi(t) + f\big(t,\xi(t)+\mathcal{R}(t)u(t)\big) + \mathcal{A}(t)\mathcal{R}(t)u(t) - \dot{\mathcal{R}}(t)u(t) - \mathcal{R}(t) \dot{u}(t). 
\end{align}
In view of Condition~\ref{ass:1, bdry} and Lemma~\ref{lm:pruess-extended}, this transformed evolution equation~\eqref{eq:transformed-system} has a unique maximal classical solution for every initial state $\xi_s \in D(A(s))$ at time $s \in \R^+_0$. And therefore, by the aforementioned one-to-one correspondence, the assertion of the lemma follows.
\end{proof}

In the entire Section~\ref{sect:wp-unbd-io-op}, the symbols $\mathcal{D}_s$ and $x(\cdot,s,x_s,u)$ will have the meaning from the above lemma. Also, we will write $x(\cdot,x_0,u) := x(\cdot,0,x_0,u)$ for brevity.  
In order to obtain globality of the maximal classical solutions from above, we make the following additional assumptions.

\begin{cond} \label{ass:2, bdry}
\begin{itemize}
\item[(i)] System~\eqref{eq:wp-semilin-bdry-contr/obs} is scattering-passive w.r.t.~a continuously differentiable storage function $V$, that is, $V \in C^1(\R^+_0\times X,\R^+_0)$ and for some $\alpha, \beta > 0$
\begin{gather}
\dot{V}(t,x(t,s,x_s,u)) \le \alpha \norm{u(t)}_U^2 - \beta \norm{y(t,s,x_s,u)}_Y^2
\qquad (t \in [s,T_{s,x_s,u}))
\label{eq:scattering-passive, bdry}
\\
(y(t,s,x_s,u) := \mathcal{C}(t)x(t,s,x_s,u)) \notag
\end{gather}
for every $s \in \R^+_0$ 
and every $(x_s,u) \in \mathcal{D}_s$
\item[(ii)] $V(t,\cdot)$ is equivalent to the norm $\norm{\cdot}$ of $X$ uniformly w.r.t.~$t$, that is, for some $\ul{\psi}, \ol{\psi} \in \mathcal{K}_{\infty}$
\begin{align} \label{eq:V-equiv-to-norm, bdry}
\ul{\psi}(\norm{x}) \le V(t,x) \le \ol{\psi}(\norm{x}) 
\qquad ((t,x) \in \R^+_0 \times X).
\end{align} 
\end{itemize}
\end{cond}

\begin{lm} \label{lm:glob-ex-UGS,bdry}
If Condition~\ref{ass:1, bdry} and~\ref{ass:2, bdry} are satisfied, then the maximal classical solution $x(\cdot,s,x_s,u)$ 
exists globally in time for every $s \in \R^+_0$ and $(x_s,u) \in \mathcal{D}_s$, that is, $T_{s,x_s,u} = \infty$. Additionally, there exist $\sigma, \gamma \in \mathcal{K}$ such that
\begin{align} \label{eq:UGS, bdry}
\norm{x(t,s,x_s,u)} \le \sigma(\norm{x_s}) + \gamma(\norm{u}_{[s,t],2})
\qquad (t \in [s,\infty))
\end{align}
for every $s \in \R^+_0$ and $(x_s,u) \in \mathcal{D}_s$.
\end{lm}

\begin{proof}
We first show that there exist $\sigma, \gamma \in \mathcal{K}$ such that the estimate~\eqref{eq:UGS, bdry} is satisfied at least for all $t \in [s,T_{s,x_s,u})$. Indeed, this follows in exactly the same way as in the case of bounded in- and output operators (Lemma~\ref{lm:glob-ex-UGS,distr}). 
With this estimate and the compactness of $\supp u$, in turn, we obtain
\begin{align*}
\sup_{t \in [s,T_{s,x_s,u})} \norm{x(t,s,x_s,u)-\mathcal{R}(t)u(t)} 
\le \sigma(\norm{x_s}) &+ \gamma(\norm{u}_{[s,\infty),2}) \notag \\
&+ \sup_{t\in\supp u} \norm{\mathcal{R}(t)} \norm{u}_{[s,\infty),\infty} < \infty
\end{align*}
for every $s \in \R^+_0$ and $(x_s,u) \in \mathcal{D}_s$. 
So, as $\xi := x(\cdot,s,x_s,u)-\mathcal{R}(\cdot)u(\cdot)$ is a maximal classical solution of~\eqref{eq:transformed-system} by the one-to-one correspondence between~\eqref{eq:original-system} and~\eqref{eq:transformed-system}, it is even global by virtue of Lemma~\ref{lm:pruess-extended}. In other words, $T_{s,x_s,u} = \infty$, as desired. 
\end{proof}

\subsection{Well-posedness: generalized solutions and outputs}

With the above preliminaries on classical solutions, we can now move on to establish also the unique existence of generalized solutions and outputs. In order to do so, we exploit the following integral equation for classical solutions of~\eqref{eq:wp-semilin-bdry-contr/obs}:
\begin{align} \label{eq:int-eq-bdry}
x(t,x_0,u) = T(t,0)x_0 + \int_0^t T(t,s)f(s,x(s,x_0,u))\d s + \Phi_t(u)
\end{align}
for all $t \in \R^+_0$ and $(x_0,u) \in \mathcal{D}_0$, where  
\begin{gather}
\Phi_t(u) := \int_0^t T(t,s)\big( \mathcal{A}(s)\mathcal{R}(s)u(s) - \mathcal{\dot{R}}(s)u(s) - \mathcal{R}(s) \dot{u}(s) \big) \d s \notag \\
-T(t,0)\mathcal{R}(0)u(0) + \mathcal{R}(t)u(t)
\label{eq:def-Phi-t,bdry}
\end{gather}
and where $T$ is the evolution system for $A$ on the spaces $D(A(s))$ (Lemma~\ref{lm:zeitentw-zu-A, ex-und-absch}). (Invoke Lemma~\ref{lm:pruess-extended} in conjunction with the one-to-one correspondence between~\eqref{eq:original-system} and~\eqref{eq:transformed-system} to obtain the integral equation~\eqref{eq:int-eq-bdry}.)

\begin{lm} \label{lm:dichtheit-und-Phi-t-stetig-fb,bdry}
If Condition~\ref{ass:1, bdry} is satisfied, then
\begin{itemize}
\item[(i)] $\mathcal{D}_0$ is dense in $X \times L^2_{\mathrm{loc}}(\R^+_0,U)$
\item[(ii)] $\Phi_t: C^2([0,t],U) \to X$ defined by~\eqref{eq:def-Phi-t,bdry} for every $t \in (0,\infty)$ uniquely extends to a bounded linear operator $\ol{\Phi}_t: L^2([0,t],U) \to X$  and 
\begin{align} \label{eq:absch-Phi-t,bdry}
C_{t_0} := \sup_{t\in [0,t_0]} \norm{\ol{\Phi}_t} < \infty
\end{align}
for every $t_0 \in (0,\infty)$.
\end{itemize}
\end{lm}

\begin{proof}
Assertion~(i) is again a 
consequence of the density of $C_{\mathrm{c}}^2(\R^+_0,U)$ in $L^2_{\mathrm{loc}}(\R^+_0,U)$ (Lemma~\ref{lm:density-and-choice-of-lipschitz-constants}) and the density of $D(A(0)) = M(0)^{-1}D(A_0(0))$ in $X$ (Condition~\ref{ass:1, bdry}(i)) in conjunction with the alternative description~\eqref{eq:alternative-def-of-classical-dataset,bdry} of $\mathcal{D}_0$. 
Indeed, let $(x_0,u) \in X \times L^2_{\mathrm{loc}}(\R^+_0,U)$. Then there exists a sequence $(u_n)$ in $C_{\mathrm{c}}(\R^+_0,U)$ with 
\begin{align*}
u_n \underset{L^2_{\mathrm{loc}}(\R^+_0,U)}{\longrightarrow} u \qquad (n \to \infty).
\end{align*}
Since $D(A(0))$ is dense in $X$, so is the subspace $\mathcal{R}(0)u_n(0) + D(A(0))$ and therefore for every $n \in \N$ there exists an $x_{0n} \in X$ with
\begin{align*}
x_{0n} \in \mathcal{R}(0)u_n(0) + D(A(0)) 
\qquad \text{and} \qquad
x_{0n} \in \ol{B}_{1/n}^{X}(x_0).
\end{align*}
Consequently, $(x_{0n},u_n) \in \mathcal{D}_0$ 
and $(x_{0n},u_n)  \underset{X \times L^2_{\mathrm{loc}}(\R^+_0,U)}{\longrightarrow} (x_0,u)$ as $n \to \infty$, as desired. 

Assertion~(ii) now does not immediately follow from the definition of $\Phi_t$ anymore, but from~\eqref{eq:int-eq-bdry} instead. 
Indeed, let $t_0 \in (0,\infty)$ be fixed and let $t \in [0,t_0]$ and
\begin{align} \label{eq:Phi-t-bd,bdry}
u_t \in C^2([0,t],U) \qquad \text{with} \qquad \norm{u_t}_{[0,t],2} \le 1. 
\end{align}
Then, of course, there exists a $u \in C_c^2(\R^+_0,U)$ with $u|_{[0,t]} = u_t$ and, by the density of $\mathcal{R}(0)u(0) + D(A(0))$ in $X$, there also exists an $x_0 \in X$ with
\begin{align*}
x_0 \in \mathcal{R}(0)u(0) + D(A(0)) 
\qquad \text{and} \qquad
x_0 \in \ol{B}_{1}^{X}(0).
\end{align*}
Consequently, $(x_0,u) \in \mathcal{D}_0$. 
We can thus conclude from~\eqref{eq:int-eq-bdry} with the help of~\eqref{eq:UGS, bdry}, \eqref{eq:zeitentw-zu-A,absch}, and Condition~\ref{ass:1, bdry}(iii) that
\begin{align} \label{eq:absch-Phi-t,bdry,beweis}
\norm{\Phi_t(u_t)} &= \norm{\Phi_t(u)} 
\le \rho_{t_0}(x_0,u) + M_{t_0}\e^{\omega_{t_0} t_0} \norm{x_0} + M_{t_0}\e^{\omega_{t_0} t_0} \cdot \notag \\
&\qquad \qquad \qquad \cdot \int_0^t L_{t_0+\rho_{t_0}(x_0,u)}  (t_0+\rho_{t_0}(x_0,u)) + \norm{f(0,0)} \d s \notag \\
&\le \rho_0 + M_{t_0}\e^{\omega_{t_0} t_0} \Big( \norm{x_0} + L_{t_0+\rho_0} (t_0+\rho_0)t_0 + \norm{f(0,0)}t_0 \Big),
\end{align}
where we used that 
\begin{align*}
\rho_{t_0}(x_0,u) := \sigma(\norm{x_0}) + \gamma(\norm{u}_{[0,t_0],2}) \le \sigma(1) + \gamma(1) =: \rho_0
\end{align*}
and that the Lipschitz constants $L_{\rho}$, chosen according to Lemma~\ref{lm:density-and-choice-of-lipschitz-constants}, are monotonically increasing in $\rho$. And from~\eqref{eq:absch-Phi-t,bdry,beweis} and~\eqref{eq:Phi-t-bd,bdry}, 
in turn, the assertion~(ii) is clear. 
\end{proof}

With the assumptions and preparations made so far, we can already prove the existence of unique generalized solutions and their continuous dependence on the data. In order to get the same things also for  generalized outputs, we impose the following final assumption.

\begin{cond}  \label{ass:3, bdry}
\begin{itemize}
\item[(i)] $\partial V: \R^+_0 \times X \to L(\R\times X, \R)$, the derivative of $V$, is bounded on bounded subsets of $\R^+_0 \times X$
\item[(ii)] $y(\cdot,s,x_s,u) := \mathcal{C}(\cdot)x(\cdot,s,x_s,u)$ is measurable for every $s \in \R^+_0$ and $(x_s,u) \in \mathcal{D}_s$.
\end{itemize}
\end{cond}

\begin{thm} \label{thm:wp, bdry}
If Condition~\ref{ass:1, bdry}, \ref{ass:2, bdry} and~\ref{ass:3, bdry} are satisfied and if  $f(t,0) =0$ for every  $t \in \R^+_0$, then the system~\eqref{eq:wp-semilin-bdry-contr/obs} is well-posed
and uniformly globally stable.
\end{thm}

\begin{proof}
(i) We first show that for every $t_0 \in (0,\infty)$ and every $(x_{01},u_1), (x_{02},u_2) \in \mathcal{D}_0$ one has the following fundamental estimate:
\begin{align} \label{eq:bew,general-solution,bdry,zentrale-absch}
\|x(\cdot,x_{01},u_1) &- x(\cdot,x_{02},u_2)\|_{[0,t_0],\infty} 
\le 
\Big( M_{t_0}\e^{\omega_{t_0} t_0} \norm{x_{01}-x_{02}} + C_{t_0} \norm{u_1-u_2}_{[0,t_0],2} \Big) \cdot \notag \\
&\quad \cdot \exp\big( M_{t_0}\e^{\omega_{t_0} t_0} L_{t_0+\rho_{t_0}(x_{01},u_1,x_{02},u_2)} t_0 \big), 
\end{align}
where $M_{t_0}$, $\omega_{t_0}$ are as in Lemma~\ref{lm:zeitentw-zu-A, ex-und-absch}, 
$C_{t_0}$ is as in Lemma~\ref{lm:dichtheit-und-Phi-t-stetig-fb,bdry}, 
$L_{\rho}$ are Lipschitz constants chosen as in Lemma~\ref{lm:density-and-choice-of-lipschitz-constants} and 
\begin{align} \label{eq:def-rho-vier-argumente,bdry}
\rho_{t_0}(x_{01},u_1,x_{02},u_2) := \sigma(\norm{x_{01}}) + \gamma(\norm{u_1}_{[0,t_0],2}) + \sigma(\norm{x_{02}}) + \gamma(\norm{u_2}_{[0,t_0],2})
\end{align} 
with $\sigma, \gamma$ as in Lemma~\ref{lm:glob-ex-UGS,bdry}. 
So let $t_0 \in (0,\infty)$ and $(x_{01},u_1)$, $(x_{02},u_2) \in \mathcal{D}_0$ and write $x_i := x(\cdot,x_{0i},u_i)$. 
It then follows from~\eqref{eq:int-eq-bdry} with the help of~\eqref{eq:zeitentw-zu-A,absch}, \eqref{eq:absch-Phi-t,bdry}, and the Lipschitz continuity of $f$ on bounded subsets (Condition~\ref{ass:1, bdry}(iii)) combined with~\eqref{eq:UGS, bdry} that
\begin{align}
\norm{x_1(t)-x_2(t)} &\le M_{t_0}\e^{\omega_{t_0} t_0} \norm{x_{01}-x_{02}} + C_{t_0} \norm{u_1-u_2}_{[0,t_0],2} \notag \\
&\quad \qquad + M_{t_0}\e^{\omega_{t_0} t_0} \int_0^t L_{t_0+\rho_{t_0}(x_{01},u_1,x_{02},u_2)} \norm{x_1(s)-x_2(s)} \d s
\end{align}
for all $t \in [0,t_0]$. And therefore the desired estimate~\eqref{eq:bew,general-solution,bdry,zentrale-absch} follows by virtue of Gr\"onwall's lemma. 
\smallskip

Combining this estimate~\eqref{eq:bew,general-solution,bdry,zentrale-absch} now with the density of $\mathcal{D}_0$ in $X\times L^2_{\mathrm{loc}}(\R^+_0,U)$ (Lem-ma~\ref{lm:dichtheit-und-Phi-t-stetig-fb,bdry}), we immediately 
see that for every $(x_0,u) \in X\times L^2_{\mathrm{loc}}(\R^+_0,U)$ there exists a unique generalized solution 
\begin{align} \label{eq:generalized-sol,bdry}
x(\cdot,x_0,u) \in C(\R^+_0,X)
\end{align}
to~\eqref{eq:wp-semilin-bdry-contr/obs}. (As for the uniqueness of these generalized solutions, the same remarks apply as those made after~\eqref{eq:generalized-sol,distr} for the case of bounded in- and output operators.)
Since the right-hand side of the estimate~\eqref{eq:UGS, bdry} with $s=0$ depends continuously on $(x_0,u)$, this estimate extends from $\mathcal{D}_0$ to arbitrary $(x_0,u) \in X \times L^2_{\mathrm{loc}}(\R^+_0,U)$. Consequently, \eqref{eq:wp-semilin-bdry-contr/obs} is uniformly globally stable.  
Since, moreover, the right-hand side of~\eqref{eq:bew,general-solution,bdry,zentrale-absch} depends continuously on $(x_{01},u_1), (x_{02},u_2)$, 
this estimate extends from $\mathcal{D}_0$ to arbitrary $(x_{01},u_1), (x_{02},u_2) \in X\times L^2_{\mathrm{loc}}(\R^+_0,U)$. And this extended estimate, 
in turn, yields the continuity of the generalized solution map $(x_0,u) \mapsto x(\cdot,x_0,u) \in C(\R^+_0,X)$. 
\smallskip

(ii) We first show that for every $t_0 \in (0,\infty)$ and every $(x_{01},u_1), (x_{02},u_2) \in \mathcal{D}_0$ one has the following fundamental estimate:
\begin{align} \label{eq:bew,general-output,bdry,zentrale-absch}
&\beta \norm{y(\cdot,x_{01},u_1)-y(\cdot,x_{02},u_2)}_{[0,t_0],2}^2 
\le 
\ol{\psi}(\norm{x_{01}-x_{02}}) + \alpha \norm{u_1-u_2}_{[0,t_0],2}^2 \notag \\
&\qquad \qquad \qquad \, + 2 M_{t_0+\rho_{t_0}(x_{01},u_1,x_{02},u_2)} \norm{x(\cdot,x_{01},u_1)-x(\cdot,x_{02},u_2)}_{[0,t_0],\infty} t_0,  
\end{align}
where $\rho_{t_0}(x_{01},u_1,x_{02},u_2)$ is defined as in~\eqref{eq:def-rho-vier-argumente,bdry} 
and where 
$M_{\rho} := K_{\rho} L_{\rho}$ with 
\begin{align*}
K_{\rho} \ge  K_{\rho}^0 := \sup \{ \norm{\partial V(t,x)}: t + \norm{x} \le \rho \}
\end{align*}
chosen such that $\rho \mapsto K_{\rho}$ is continuous and monotonically increasing (see the proof of Lemma~\ref{lm:density-and-choice-of-lipschitz-constants}) and with Lipschitz constants $L_{\rho}$ chosen as in Lemma~\ref{lm:density-and-choice-of-lipschitz-constants}. 
%
So let $t_0 \in (0,\infty)$ and $(x_{01},u_1)$, $(x_{02},u_2) \in \mathcal{D}_0$ and write $x_i := x(\cdot,x_{0i},u_i)$ and $y_i := y(\cdot,x_{0i},u_i) = \mathcal{C}(\cdot)x(\cdot,x_{0i},u_i)$ as well as $x_{12} := x_1-x_2$ and $u_{12} := u_1-u_2$.  
It then follows by the differential equation~\eqref{eq:wp-semilin-bdry-contr/obs}  that
\begin{align} \label{eq:bew,general-output,bdry,1}
\frac{\d}{\d s} V(s,x_{12}(s)) &= \partial_s V(s,x_{12}(s)) + \partial_x V(s,x_{12}(s))\big( \mathcal{A}(s)x_{12}(s) + f(s,x_{12}(s)) \big) \notag \\
&\quad + \partial_x V(s,x_{12}(s))\big( f(s,x_1(s))-f(s,x_2(s)) - f(s,x_{12}(s)) \big) 
\end{align} 
for all $s\in \R^+_0$. Since by the classical solution property of $x_i$ for~\eqref{eq:wp-semilin-bdry-contr/obs}
\begin{align*}
x_{12}(s) \in D(\mathcal{A}(s)) \qquad \text{with} \qquad \mathcal{B}(s)x_{12}(s) = u_{12}(s) \qquad \text{and} \qquad u_{12} \in C_c^2(\R^+_0,U),
\end{align*}
we have $(x_{12s},u_{12}) := (x_{12}(s),u_{12}) \in \mathcal{D}_s$ and thus it follows by~\eqref{eq:scattering-passive, bdry} that the first part of the right-hand side of~\eqref{eq:bew,general-output,bdry,1} can be estimated as follows: 
\begin{align} \label{eq:bew,general-output,bdry,2}
\partial_s V(s,x_{12}(s)) &+ \partial_x V(s,x_{12}(s))\big( \mathcal{A}(s)x_{12}(s) + f(s,x_{12}(s)) \big) \notag \\
&= \frac{\d}{\d t} V\big(t,x(t,s,x_{12s},u_{12})\big) \big|_{t=s}
\le \alpha \norm{u_{12}(s)}_U^2 - \beta \norm{\mathcal{C}(s)x_{12s}}_Y^2 \notag \\
&= \alpha \norm{u_1(s)-u_2(s)}_U^2 - \beta \norm{y_1(s)-y_2(s)}_Y^2 
\qquad (s \in \R^+_0). 
\end{align}
Since, moreover, $\partial V$ and $f$ are bounded or Lipschitz, respectively, on bounded subsets of $\R^+_0 \times X$ (Condition~\ref{ass:3, bdry}(i) and~\ref{ass:1, bdry}(iii)!) and $f(s,0)=0$, it further follows by~\eqref{eq:UGS, bdry} that the second part of the right-hand side of~\eqref{eq:bew,general-output,bdry,1} can be estimated as follows: 
\begin{align} \label{eq:bew,general-output,bdry,3}
\partial_x V(s,x_{12}(s)) & \big( f(s,x_1(s)) - f(s,x_2(s)) - f(s,x_{12}(s)) \big) \notag \\
&\qquad \le 2 M_{t_0 + \rho_{t_0}(x_{01},u_1,x_{02},u_2)} \norm{x_1(s)-x_2(s)}
\qquad (s\in\R^+_0). 
\end{align} 
Inserting now~\eqref{eq:bew,general-output,bdry,2} and~\eqref{eq:bew,general-output,bdry,3} into~\eqref{eq:bew,general-output,bdry,1} and integrating the resulting estimate (Condition~\ref{ass:3, bdry}(ii)!), 
we finally obtain 
the desired estimate~\eqref{eq:bew,general-output,bdry,zentrale-absch}. 
\smallskip

Combining this estimate~\eqref{eq:bew,general-output,bdry,zentrale-absch} now with the density of $\mathcal{D}_0$ in $X\times L^2_{\mathrm{loc}}(\R^+_0,U)$ (Lem-ma~\ref{lm:dichtheit-und-Phi-t-stetig-fb,bdry}) and the continuity of 
$(x_0,u) \mapsto x(\cdot,x_0,u) \in C(\R^+_0,X)$ established in part~(i) above, we immediately 
see that for every $(x_0,u) \in X\times L^2_{\mathrm{loc}}(\R^+_0,U)$ there exists a unique generalized output $$y(\cdot,x_0,u) \in L^2_{\mathrm{loc}}(\R^+_0,Y)$$ to~\eqref{eq:wp-semilin-bdry-contr/obs}.
Since, moreover, the right-hand side of~\eqref{eq:bew,general-output,bdry,zentrale-absch} depends continuously on $(x_{01},u_1)$, $(x_{02},u_2)$, 
this estimate extends from $\mathcal{D}_0$ to arbitrary $(x_{01},u_1), (x_{02},u_2) \in X\times L^2_{\mathrm{loc}}(\R^+_0,U)$. And this extended estimate, 
in turn, yields the continuity of the generalized output map $(x_0,u) \mapsto y(\cdot,x_0,u) \in L^2_{\mathrm{loc}}(\R^+_0,Y)$. 
\end{proof}

\section{Some applications} \label{sect:applications}

We now present two classes of applications of our abstract well-posedness and stability results: 
one for the case of bounded control and observation operators and one for the case of unbounded control and observation operators. 
%
In both cases, the considered systems arise as closed-loop systems by coupling a linear system $\ul{\mathfrak{S}}$ to a nonlinear controller $\mathfrak{S}_c$ with a standard feedback interconnection, that is, the output $\ul{y}$ of the linear system 
is the input $u_c$ of the controller, the input $\ul{u}$ of the linear system is minus the output $-y_c$ of the controller plus the (external) input $u$ of the closed-loop system, and $\ul{y}$ is also the (external) output of the closed-loop system. In short,
\begin{align} \label{eq:standard-feedback-interconnection}
\ul{y}(t) = u_c(t) \qquad \text{and} \qquad -y_c(t)+u(t) = \ul{u}(t) \qquad \text{and} \qquad \ul{y}(t) = y(t)
\end{align}
and in pictures such a closed-loop system 
can be represented as in the figure below. 
\begin{figure}[h!]
\centering
\begin{tikzpicture}[scale =1]
\tikzstyle{tf} = [rectangle, draw = black, minimum width= 1.618cm, minimum height = 1cm]
\tikzstyle{sum} = [draw=black, shape=circle, minimum width = .5, minimum height = .5]

\node(BC) [tf,scale=1, minimum width = 1cm] at (4,4){System $\ul{\mathfrak{S}}$};
\node(NL) [tf,scale=1, minimum width = 1cm] at (4,2){Controller $\mathfrak{S}_c$};
\node(s1) [sum, scale = .7, minimum width = .5] at (2,4) {$+$};
\node(m) [ right] at (2,3.5){$-$};

\node(u-ext) [above] at (1.2,4) {$u$};
\node(u) [signal, above] at (2.6,4) {$\ul{u}$};
\node(y) [above left] at (6,4){} ;
\node(o)[signal, above] at (5.6,4) {$\ul{y}$};
\node(uc) [above left] at (6,2) {$u_c$};
\node(yc) [above right] at (2,2) {$y_c$};
\node(y-ext) [above] at (6.7,4) {$y$};

\draw [->] (0.6,4)--(s1);
\draw [->] (s1)--(BC);
\draw [->] (BC)--(7.4,4);
\draw [->] (y.south east) -- (uc.south east) -- (NL);
\draw [->] (NL.west) -| (s1.south); 
\end{tikzpicture}
\end{figure}
Also, in both cases, the considered systems will be even strictly impedance-passive (instead of only scattering-passive) w.r.t.~a continuously differentiable storage function, that is,  $V \in C^1(\R^+_0 \times X,\R^+_0)$ and $U = Y$ is a Hilbert space such that for some $\varsigma > 0$
\begin{gather}
\dot{V}(t, x(t,s,x_s,u)) \le \Re \scprd{u(t),y(t)}_U - \varsigma \norm{y(t,s,x_s,u)}_U^2 \label{eq:imp-passive}
\qquad (s \in [s,T_{s,x_s,u})) \\
(y(t,s,x_s,u) := \mathcal{C}(t)x(t,s,x_s,u)) \notag
\end{gather}
for every $s\in \R^+_0$  and $(x_s,u) \in \mathcal{D}_s$ 
with the classical data set $\mathcal{D}_s$ from~\eqref{eq:D_s-distr-def} or~\eqref{eq:D_s-bdry-def} respectively.
Clearly, this implies the scattering-passivity estimates~\eqref{eq:scattering-passive, distr} and~\eqref{eq:scattering-passive, bdry} with $\alpha := \frac{1}{2\varsigma}$ and $\beta := \frac{\varsigma}{2}$.

\subsection{Case with bounded control and observation operators} \label{sect:appl-bd}

\subsubsection{Setting: open-loop system and controller} \label{sect:appl-bd,open-loop-and-controller}

As our open-loop system, we consider a non-autonomous linear collocated system with bounded control and observation operators. Such a system evolves according to the differential equation
\begin{align} \label{eq:open-loop-diff-eq,distr}
\dot{x}(t) = \mathcal{A}(t)x(t) + \mathcal{B}(t)\ul{u}(t)
\end{align}
in the state space $X$ with the additional observation condition
\begin{align} \label{eq:open-loop-output,distr}
\ul{y}(t) = \mathcal{C}(t)x(t).
\end{align}
In these equations, $A(t) := \mathcal{A}(t)$ are operators as in Lemma~\ref{lm:vereinf-vor-A(t)} (in particular, $X$ is a Hilbert space) and $\mathcal{B}(t) \in L(U,X)$, $\mathcal{C}(t) \in L(X,Y)$ with a Hilbert space $U=Y$ such that
\begin{align} \label{eq:collocation-condition}
\mathcal{C}(t) = \mathcal{B}(t)^*M(t).
\end{align}
(In concrete examples, 
\eqref{eq:collocation-condition} typically means~\cite{Oo00} that the observation takes place at the same location as the control or, in other words, that control and observation are collocated).
Additionally, $t \mapsto M(t)$ from Lemma~\ref{lm:vereinf-vor-A(t)} is assumed to be monotonically decreasing, that is, for every $x \in X$
\begin{align} \label{eq:M(t)-mon-decreasing} 
\scprd{M(t)x,x}_X \le \scprd{M(s)x,x}_X \qquad (s\le t) 
\end{align}
and $t \mapsto \mathcal{B}(t)$ is assumed to be locally Lipschitz. 
%
We now couple our open-loop system~\eqref{eq:open-loop-diff-eq,distr}-\eqref{eq:open-loop-output,distr} to a nonlinear static controller described by the input-output relation
\begin{align} \label{eq:controller,distr}
y_c(t) = g(u_c(t)),
\end{align}
where $g: U \to U$ is Lipschitz on bounded subsets of $U$ and strictly damping in the sense that for some $\varsigma > 0$
\begin{align} \label{eq:g-damping,distr}
\scprd{y,g(y)}_U \ge \varsigma \norm{y}_U^2 \qquad (y\in U). 
\end{align}

\subsubsection{Closed-loop system}

Choosing 
the coupling of the controller~\eqref{eq:controller,distr} to the open-loop system~\eqref{eq:open-loop-diff-eq,distr}-\eqref{eq:open-loop-output,distr} to be a standard feedback interconnection~\eqref{eq:standard-feedback-interconnection}, 
we see that the arising closed-loop system is described by a differential equation of the form
\begin{align} \label{eq:cl-system-diff-eq,distr}
\dot{x}(t) = \mathcal{A}(t)x(t) + f(t,x(t)) + \mathcal{B}(t)u(t)
\end{align}
in the state space $X$ with the additional observation condition
\begin{align} \label{eq:cl-system-output-cond,distr}
y(t) = \mathcal{C}(t)x(t),
\end{align}
where $\mathcal{A}(t)$, $\mathcal{B}(t)$, $\mathcal{C}(t)$ are as above and $f: \R^+_0 \times X \to X$ is defined by
\begin{align*}
f(t,x) := -\mathcal{B}(t)\, g\big(\mathcal{B}(t)^*M(t)x\big).
\end{align*}

\begin{cor} \label{cor:appl-distr}
With the above assumptions, the closed-loop system~\eqref{eq:cl-system-diff-eq,distr}-\eqref{eq:cl-system-output-cond,distr} is well-posed and uniformly globally stable. 
\end{cor}

\begin{proof}
We verify the assumptions of Theorem~\ref{thm:wp, distr}. 
As a first step, we observe that Condition~\ref{ass:1, distr} is satisfied.
Indeed, this immediately follows from our assumptions above and Lemma~\ref{lm:vereinf-vor-A(t)}.
\smallskip

As a second step, we show that Condition~\ref{ass:2, distr} is satisfied with
\begin{align}
V(t,x) := \frac{1}{2} \scprd{M(t)x,x}_X \qquad ((t,x) \in \R^+_0\times X). 
\end{align} 
Indeed, $V \in C^1(\R^+_0\times X,\R^+_0)$ and for every $s \in \R^+_0$ and $(x_s,u) \in \mathcal{D}_s$ we see with $x(t) := x(t,s,x_s,u)$ and $y(t) := \mathcal{C}(t)x(t,s,x_s,u)$ that
\begin{align}
\frac{\d}{\d t} V(t,x(t)) 
&= \frac{1}{2} \langle \dot{M}(t)x(t),x(t) \rangle_X + \Re \scprd{M(t)x(t), A(t)x(t)}_X \notag \\
&\quad - \Re \scprd{M(t)x(t),\mathcal{B}(t) \, g(\mathcal{B}(t)^*M(t)x(t))}_X + \Re \scprd{M(t) x(t),\mathcal{B}(t)u(t)}_X \notag \\
&\le -\Re \scprd{\mathcal{C}(t)x(t),g(\mathcal{C}(t)x(t))}_U + \Re \scprd{\mathcal{C}(t)x(t),u(t)}_U \notag \\
&\le \Re \scprd{u(t),y(t)}_U - \varsigma \norm{y(t)}_U^2
\end{align}
for all $t \in [s,T_{s,x_s,u})$. (In the first inequality, we used the monotonicity~\eqref{eq:M(t)-mon-decreasing} and the contraction semigroup generation assumption from Lemma~\ref{lm:vereinf-vor-A(t)} and in the second inequality we used the strict damping assumption~\eqref{eq:g-damping,distr}.)
Consequently, Condition~\ref{ass:2, distr}(i) is satisfied with $\alpha := \frac{1}{2\varsigma}$ and $\beta := \frac{\varsigma}{2}$.
And in view of~\eqref{eq:M(t)-lower-and-upper-bounded}, Condition~\ref{ass:2, distr}(ii) is satisfied as well.
\smallskip

As a third step, we observe that Condition~\ref{ass:3, distr} is satisfied. 
Condition~\ref{ass:3, distr}(i) is obviously satisfied and in order to verify 
Condition~\ref{ass:3, distr}(ii) we have only to use that $t \mapsto \mathcal{C}(t) = \mathcal{B}(t)^*M(t)$ is locally Lipschitz continuous. 
\end{proof}

We close with a concrete example. In spite of the boundary control and observation used there, it can be formulated as a system with bounded in- and output operators. 

\begin{ex}
Consider an Euler-Bernoulli beam with possibly time-dependent flexural rigidity, that is, the transverse displacement $w(t,\zeta)$ of the beam at position $\zeta \in (a,b)$ is governed by the partial differential equation
\begin{align} \label{eq:E-B-beam}
\rho \, \partial_t^2 w(t,\zeta) = \lambda(t) \, \partial_{\zeta}^4 w(t,\zeta)
\end{align}
for $t \in \R^+_0$ and $\zeta \in (a,b)$, where $\rho \in (0,\infty)$ is the mass density and $\lambda(t) \in (0,\infty)$ is the flexural rigidity of the beam. We assume that the beam is clamped at its left end, that is,
\begin{align}
w(t,a) = 0 \qquad \text{and} \qquad \partial_{\zeta}w(t,a) = 0 
\end{align}
and that a point mass (with mass $m$ and moment of inertia $J$) is attached at the right end of the beam. Also, a piezoelectric film is bonded to the beam which applies a bending moment to the beam when a voltage is applied to it. We assume that no external force acts on the tip mass, that is,
\begin{align} 
m \, \partial_t^2 w(t,b) - \lambda(t) \, \partial_{\zeta}^3 w(t,b) = 0,
\end{align}
and that the torque acting on the tip mass is given by the voltage $\ul{u}(t)$ of the piezoelectric film (control input) multiplied with a modulating prefactor $\kappa(t)$, that is,
\begin{align}
J \, \partial_t^2 \partial_{\zeta} w(t,b) + \lambda(t) \, \partial_{\zeta}^2 w(t,b) = \kappa(t) \ul{u}(t).
\end{align}
As the measurement $\ul{y}(t)$ (observation output), we take the angular velocity at the right end of the beam multiplied with the same modulating prefactor $\kappa(t)$, that is,
\begin{align} \label{eq:output-E-B-beam}
\ul{y}(t) = \kappa(t) \, \partial_t \partial_{\zeta} w(t,b).
\end{align}
And finally, the modulation function $t \mapsto \kappa(t)$ is assumed to be twice continuously differentiable, while the flexural rigidity $t \mapsto \lambda(t)$ is assumed to be monotonically decreasing and twice continuously differentiable such that
\begin{align*}
\ul{m} \le \lambda(t) \le \ol{m} \qquad (t \in \R^+_0).
\end{align*}
for some constants $\ul{m}, \ol{m} \in (0,\infty)$. 
See, for instance, \cite{MiSt16}, \cite{CuZw16} or \cite{Sl89} for autonomous versions of this linear system. 
Writing 
\begin{equation} \label{eq:E-B-beam-reformulation-1}
\begin{gathered} 
x_1(t)(\zeta) := w(t,\zeta), \qquad x_2(t)(\zeta) := \rho \, \partial_t w(t,\zeta), \qquad x_3(t) := \partial_t w(t,b), \\
x_4(t) := \partial_t \partial_{\zeta} w(t,b),
\end{gathered}
\end{equation}
we can bring the system~\eqref{eq:E-B-beam}-\eqref{eq:output-E-B-beam} to the form~\eqref{eq:open-loop-diff-eq,distr}-\eqref{eq:open-loop-output,distr} of a linear system with bounded in- and output operators. Indeed, let $\mathcal{A}(t) := \mathcal{A}_0 M(t)$ with
\begin{align} \label{eq:E-B-beam-reformulation-2}
\mathcal{A}_0 x := 
\begin{pmatrix}
x_2/\rho \\ -\partial_{\zeta}^4 x_1 \\ \partial_{\zeta}^3 x_1(b)/m \\ -\partial_{\zeta}^2 x_1(b)/J
\end{pmatrix}
\qquad \text{and} \qquad 
M(t) 
:= 
\begin{pmatrix}
\lambda(t) & 0 & 0 & 0 \\
0 & 1 & 0 & 0 \\
0 & 0 & 1 & 0 \\
0 & 0 & 0 & 1
\end{pmatrix}
\end{align}
for $x = (x_1,x_2,x_3,x_4) \in D(\mathcal{A}_0)$ in the state space $X$, where
\begin{gather*}
D(\mathcal{A}_0) := \big\{ (x_1,x_2,x_3,x_4) \in H^4((a,b),\R) \times H^2((a,b),\R) \times \R \times \R: x_1(a), \partial_{\zeta}x_1(a) = 0, \\
x_2(a), \partial_{\zeta}x_2(a) = 0 \text{ and } x_3 = x_2(b)/\rho \text{ and } x_4 = \partial_{\zeta}x_2(b)/\rho \big\}, \\
X := \big\{ (x_1,x_2,x_3,x_4) \in H^2((a,b),\R) \times L^2((a,b),\R) \times \R \times \R: x_1(a), \partial_{\zeta}x_1(a) = 0 \big\}
\end{gather*}
and where the state space $X$ is endowed with the scalar product $\scprd{\cdot,\cdot \cdot}_X$ defined by
\begin{align*}
\scprd{x,\tilde{x}}_X := \int_a^b \partial_{\zeta}^2 x_1(\zeta) \partial_{\zeta}^2 \tilde{x}_1(\zeta) \d \zeta + \frac{1}{\rho} \int_a^b x_2(\zeta) \tilde{x}_2(\zeta) \d \zeta + m \, x_3 \tilde{x}_3 + J \, x_4 \tilde{x}_4.
\end{align*}
Also, let $U, Y := \R$ and 
\begin{align} \label{eq:E-B-beam-reformulation-3}
\mathcal{B}(t)u := \frac{\kappa(t)}{J} \begin{pmatrix} 0 \\ 0 \\ 0 \\ u\end{pmatrix}
\qquad \text{and} \qquad
\mathcal{C}(t)x := \kappa(t) x_4
\end{align}
for $u \in U$ and $x = (x_1,x_2,x_3,x_4) \in X$. With these choices~\eqref{eq:E-B-beam-reformulation-1}-\eqref{eq:E-B-beam-reformulation-3}, it is straightforward to verify that the pde system~\eqref{eq:E-B-beam}-\eqref{eq:output-E-B-beam} is indeed a linear system of the form~\eqref{eq:open-loop-diff-eq,distr}-\eqref{eq:open-loop-output,distr} and that it satisfies all the assumptions -- in particular, the collocation and monotonicity conditions \eqref{eq:collocation-condition} and \eqref{eq:M(t)-mon-decreasing} -- on the open-loop system from Section~\ref{sect:appl-bd,open-loop-and-controller}. (In order to see that $A(t) := \mathcal{A}(t) = \mathcal{A}_0 M(t)$ indeed satisfies the assumptions of Lemma~\ref{lm:vereinf-vor-A(t)}, notice that $\mathcal{A}_0$ is a multiplicative perturbation of the contraction semigroup generator $\mathcal{A}_{00}$ from Section~8.1 of~\cite{Sl89}, namely 
\begin{align*}
\mathcal{A}_0 = V_{m,J} \mathcal{A}_{00} V_{\rho} \quad \text{with} \quad V_{m,J} := \mathrm{diag}(1,1,1/m,1/J), \quad V_{\rho} := \mathrm{diag}(1,1/\rho,1,1),
\end{align*}
and apply multiplicative perturbation arguments in the spirit of Lemma~7.2.3 of~\cite{JacobZwart}.)
So, as soon as the controller is chosen as in Section~\ref{sect:appl-bd,open-loop-and-controller} above, the arising closed-loop system will be well-posed and uniformly globally stable by virtue of Corollary~\ref{cor:appl-distr}.~$\blacktriangleleft$
\end{ex}

\subsection{Case with unbounded control and observation operators} \label{sect:appl-unbd}

\subsubsection{Setting: open-loop system and controller} \label{sect:appl-bdry-open-loop}

As our open-loop system, we consider a non-autonomous linear port-Hamiltonian system of order $N \in \N$ on a bounded  interval $(a,b)$ with control and observation at the boundary~\cite{JaLa19}. 
Such a system evolves according to the differential equation
\begin{align} \label{eq:open-loop-diff-eq}
\dot{\ul{x}}(t) = \ul{\mathcal{A}}(t)\ul{x}(t) 
=  \ul{P}_N \partial_{\zeta}^N(\ul{\mathcal{H}}(t)\ul{x}(t)) 
+ \dotsb + \ul{P}_1 \partial_{\zeta}(\ul{\mathcal{H}}(t)\ul{x}(t)) + \ul{P}_0 \ul{\mathcal{H}}(t)\ul{x}(t)
\end{align}
in the state space $\ul{X}:= L^2((a,b),\R^m)$ with $m \in \N$ and with  the additional 
control and 
observation conditions
\begin{align} \label{eq:open-loop-input/output-eq}
\ul{u}(t) = \ul{\mathcal{B}}(t)\ul{x}(t) \quad \text{and} \quad \ul{y}(t) =\ul{\mathcal{C}}(t)\ul{x}(t).
\end{align}
In these equations, $\ul{\mathcal{A}}(t): D(\ul{\mathcal{A}}(t)) \subset \ul{X} \to \ul{X}$ is the linear operator defined by
\begin{gather}
\ul{\mathcal{A}}(t)\ul{x} := \ul{\mathcal{A}}_0 \ul{\mathcal{H}}(t)\ul{x} := \ul{P}_N \partial_{\zeta}^N(\ul{\mathcal{H}}(t)\ul{x}) 
+ \dotsb + \ul{P}_1 \partial_{\zeta}(\ul{\mathcal{H}}(t)\ul{x}) + \ul{P}_0 \ul{\mathcal{H}}(t)\ul{x} \notag \\
D(\ul{\mathcal{A}}(t)) := \big\{ \ul{x} \in \ul{X}: \ul{\mathcal{H}}(t)\ul{x} \in H^N((a,b),\R^m) \text{ and } W_{B,1} (\ul{\mathcal{H}}(t) \ul{x})|_{\partial} = 0 \big\}
\label{eq:open-loop-domain}
\end{gather}
and $\ul{\mathcal{B}}(t), \ul{\mathcal{C}}(t): D(\ul{\mathcal{A}}(t)) \subset \ul{X} \to \R^k$ are the linear boundary control and observation operators defined by
\begin{align*}
\ul{\mathcal{B}}(t)\ul{x} := \ul{\mathcal{B}}_0 \ul{\mathcal{H}}(t)\ul{x} := W_{B,2} (\ul{\mathcal{H}}(t) \ul{x})|_{\partial}
\quad \text{and} \quad
\ul{\mathcal{C}}(t)\ul{x} := \ul{\mathcal{C}}_0 \ul{\mathcal{H}}(t)\ul{x} := W_{C} (\ul{\mathcal{H}}(t) \ul{x})|_{\partial}
\end{align*}
where $W_{B,1} \in \R^{(mN-k)\times 2mN}$ and $W_{B,2}, W_C \in \R^{k \times 2mN}$ with $k \in \{1, \dots, mN\}$ and where, for a function $\ul{f} \in H^N((a,b),\R^m)$, the symbol $\ul{f}|_{\partial}$ denotes the (column) vector consisting of the boundary values of the first $N-1$ derivatives of $\ul{f}$, more precisely:
\begin{align*}
\ul{f}|_{\partial} := \big(\ul{f}(b)^{\top}, \ul{f}'(b)^{\top}, \dots, \ul{f}^{(N-1)}(b)^{\top}, \ul{f}(a)^{\top}, \ul{f}'(a)^{\top}, \dots, \ul{f}^{(N-1)}(a)^{\top} \big)^{\top} \in \R^{2mN}. 
\end{align*}
%
As usual, $\ul{P}_0,\ul{P}_1, \dots, \ul{P}_N \in \R^{m\times m}$ are matrices such that $\ul{P}_N$ is invertible and $\ul{P}_1,\dots, \ul{P}_N$ are alternately symmetric and skew-symmetric while $\ul{P}_0$ is dissipative: 
\begin{align}
\ul{P}_l^{\top} = (-1)^{l+1} \ul{P}_l \qquad (l \in \{1,\dots,N\}) 
\qquad \text{and} \qquad \ul{P}_0^{\top} + \ul{P}_0 \le 0.
\end{align} 
%
Also, $\ul{\mathcal{H}}(t)(\zeta) \in \R^{m\times m}$ are symmetric matrices for $(t,\zeta) \in \R^+_0 \times (a,b)$ satisfying the following assumptions: 
\begin{itemize}
\item there exist finite positive constants $\ul{m}, \ol{m} \in (0,\infty)$ such that
\begin{align} \label{eq:H(t)(zeta)-lower-and-upper-bounded}
\ul{m} \le \ul{\mathcal{H}}(t)(\zeta) \le \ol{m} \qquad ((t,\zeta) \in \R^+_0 \times (a,b))
\end{align}
\item $\zeta \mapsto \ul{\mathcal{H}}(t)(\zeta) \in \R^{m\times m}$ is measurable for every fixed $t \in \R^+_0$
\item $t \mapsto \ul{\mathcal{H}}(t) \in L(\ul{X})$ is twice strongly continuously differentiable and monotonically decreasing, that is, for every $\ul{x} \in \ul{X}$ 
\begin{align} \label{eq:H(t)-mon-decreasing}
\scprd{\ul{\mathcal{H}}(t)\ul{x},\ul{x}}_2 \le \scprd{\ul{\mathcal{H}}(s) \ul{x},\ul{x}}_2 
\qquad (s \le t).
\end{align}
\end{itemize}
(In these 
assumptions, 
we used the symbol $\ul{\mathcal{H}}(t)$ not only to denote the measurable function $\ul{\mathcal{H}}(t): (a,b) \to \R^{m\times m}$ but, as usual, also to denote the corresponding multiplication operator $\ul{X} \ni \ul{x} \mapsto \ul{\mathcal{H}}(t)\ul{x} \in \ul{X}$, which belongs to $L(\ul{X})$ 
by virtue of~(\ref{eq:H(t)(zeta)-lower-and-upper-bounded}.b).)
%
We further assume that the boundary matrix
\begin{align} \label{eq:W-matrix, def}
W := \begin{pmatrix}
W_{B} \\ W_C
\end{pmatrix}
\in \R^{(mN+k)\times 2mN}
\qquad \text{with} \qquad
W_B := \begin{pmatrix}
W_{B,1} \\ W_{B,2}
\end{pmatrix}
\end{align}
is a matrix of full row rank $mN+k$.
%
And finally, we assume that our open-loop system~\eqref{eq:open-loop-diff-eq}-\eqref{eq:open-loop-input/output-eq} with $\mathcal{H}(t)(\zeta) \equiv I$ (identity matrix) is impedance-passive, that is, 
\begin{align} \label{eq:open-loop-imp-passive-with-H=1}
\scprd{\ul{x},\ul{\mathcal{A}}_0 \ul{x}}_2 \le (\ul{\mathcal{B}}_0 \ul{x})^{\top} \ul{\mathcal{C}}_0 \ul{x} 
\qquad (\ul{x} \in D(\ul{\mathcal{A}}_0)).
\end{align}
Concrete examples 
of open-loop systems that satisfy all the above assumptions will be given below (Example~\ref{ex:vibrating-string} and~\ref{ex:Timoshenko-beam}).
%
%
We now couple our open-loop system~\eqref{eq:open-loop-diff-eq}-\eqref{eq:open-loop-input/output-eq} to a nonlinear dynamic 
controller 
described by the ordinary differential equation
\begin{align} \label{eq:controller-diff-eq}
\dot{v}(t) = \begin{pmatrix} \dot{v}_1(t) \\ \dot{v}_2(t) \end{pmatrix}
= 
\begin{pmatrix}
K_c v_2(t) \\
-\nabla \mathcal{P}_c(v_1(t)) - \mathcal{R}_c(t,K_c v_2(t)) + B_c u_c(t) 
\end{pmatrix} 
\end{align}
in the state space $X_c := \R^{m_c} \times \R^{m_c}$ with the additional input-output relation
\begin{align} \label{eq:controller-input-output-relation}
y_c(t) = B_c^{\top} K_c v_2(t) + S_c u_c(t). 
\end{align}
%
In these equations, $K_c \in \R^{m_c \times m_c}$, $B_c \in \R^{m_c \times k}$, $S_c \in \R^{k \times k}$ represent a generalized mass matrix, an input matrix, and a direct feedthrough matrix respectively satisfying 
$K_c > 0$ and $S_c > 0$. In particular, 
\begin{align} \label{eq:S_c pos def} 
y^{\top} S_c y \ge \varsigma \,|y|^2 \qquad (y \in \R^k),
\end{align}
where $\varsigma > 0$ is the smallest eigenvalue of $S_c$. 
%
Also, the potential energy $\mathcal{P}_c: \R^{m_c} \to \R^+_0$ is differentiable such that $\nabla \mathcal{P}_c$ is locally Lipschitz continuous and $\mathcal{P}_c(0) = 0$ and the damping function $\mathcal{R}_c: \R_0^+ \times \R^{m_c} \to \R^{m_c}$ is locally Lipschitz continuous such that $\mathcal{R}_c(t,0) = 0$ for all $t$. 
And finally, we assume that
\begin{itemize}
\item $\mathcal{P}_c$ is positive definite and radially unbounded, 
that is, $\mathcal{P}_c(v_1) > 0$ for all $v_1 \in \R^{m_c} \setminus \{0\}$ 
and $\mathcal{P}_c(v_1) \longrightarrow \infty$ as $|v_1| \to \infty$
\item $\mathcal{R}_c(t,\cdot)$ is damping, that is, $v_2^{\top} \mathcal{R}_c(t,v_2) \ge 0$ for all $(t,v_2) \in \R^+_0 \times \R^{m_c}$.
\end{itemize}

\subsubsection{Closed-loop system}

Choosing 
the coupling of the controller~\eqref{eq:controller-diff-eq}-\eqref{eq:controller-input-output-relation} to the open-loop system~\eqref{eq:open-loop-diff-eq}-\eqref{eq:open-loop-input/output-eq} to be a standard feedback interconnection~\eqref{eq:standard-feedback-interconnection}, 
we see that the arising closed-loop system is described by a differential equation of the form
\begin{align} \label{eq:cl-system-diff-eq,bdry}
\dot{x}(t) = \mathcal{A}(t)x(t) + f(t,x(t))
\end{align}
in the state space $X := \ul{X}\times X_c$ with the following additional conditions for the in- and output $u, y$ of the closed-loop system: 
\begin{align} \label{eq:cl-system-input-output-cond,bdry}
u(t) = \mathcal{B}(t)x(t) \qquad \text{and} \qquad y(t) = \mathcal{C}(t)x(t).
\end{align}
In these equations, $\mathcal{A}(t): D(\mathcal{A}(t)) \subset X \to X$ and $f: \R^+_0 \times X \to X$ are the linear and nonlinear operator defined respectively  by
\begin{align*}
\mathcal{A}(t)x 
:= 
\begin{pmatrix}
\ul{\mathcal{A}}(t)\ul{x} \\ K_c v_2 \\ -v_1 + B_c \ul{\mathcal{C}}(t)\ul{x}
\end{pmatrix}
\qquad \text{and} \qquad
f(t,x)
:=
\begin{pmatrix}
0 \\ 0 \\ v_1 - \nabla \mathcal{P}_c(v_1) - \mathcal{R}_c(t,K_c v_2)
\end{pmatrix}
\end{align*}
with $D(\mathcal{A}(t)) := D(\ul{\mathcal{A}}(t)) \times X_c$ 
and $\mathcal{B}(t), \mathcal{C}(t): D(\mathcal{A}(t)) \subset X \to \R^k$ are the linear input and output operators defined by
\begin{align*}
\mathcal{B}(t)x := \ul{\mathcal{B}}(t)\ul{x} + B_c^{\top}K_c v_2 + S_c \ul{\mathcal{C}}(t)\ul{x}
\qquad \text{and} \qquad
\mathcal{C}(t)x := \ul{\mathcal{C}}(t)\ul{x},
\end{align*} 
where $\ul{x}$ and $v$ denote the components of $x$, that is, $(\ul{x},v_1,v_2) = (\ul{x},v) = x$.

\begin{cor} \label{cor:appl-bdry}
With the above assumptions, the closed-loop system~\eqref{eq:cl-system-diff-eq,bdry}-\eqref{eq:cl-system-input-output-cond,bdry} is well-posed and uniformly globally stable. 
\end{cor}

\begin{proof}
We verify the assumptions of Theorem~\ref{thm:wp, bdry}.
As a first step, we show that Condition~\ref{ass:1, bdry} is satisfied. 
We first observe that the linear part $\mathcal{A}(t)$ and the in- and output operators $\mathcal{B}(t)$, $\mathcal{C}(t)$ of our closed-loop system~\eqref{eq:cl-system-diff-eq,bdry}-\eqref{eq:cl-system-input-output-cond,bdry}  factorize in the form
\begin{gather*}
\mathcal{A}(t) = \mathcal{A}_0 M(t), \qquad \mathcal{B}(t) = \mathcal{B}_0 M(t), \qquad \mathcal{C}(t) = \mathcal{C}_0 M(t)\\
M(t) 
:= 
\begin{pmatrix}
\ul{\mathcal{H}}(t) & 0 \\ 0 & I
\end{pmatrix},
\end{gather*}
where $I$ is the identity operator on $X_c$. 
We also observe that $\mathcal{A}_0$ and $\mathcal{B}_0$, $\mathcal{C}_0$ are the linear part and the in- and output operators of the closed-loop system~(2.16)-(2.17) from~\cite{ScZw18} with $\mathcal{H}(\zeta) \equiv I$, respectively, and that by virtue of our assumptions above the assumptions from~\cite{ScZw18} -- 
and in particular Condition~2.1 and~3.1 from~\cite{ScZw18} -- are satisfied with $\mathcal{H}(\zeta) \equiv I$. So, it follows that
\begin{itemize}
\item $A_0 := \mathcal{A}_0|_{\ker \mathcal{B}_0}$ is a contraction semigroup generator on $X$ w.r.t.~the scalar product $\scprd{\cdot,\cdot \cdot}_X$ defined by $\scprd{x,y}_X := \scprd{\ul{x},\ul{y}}_2 + v_1^{\top}w_1 + v_2^{\top}K_c w_2$ (Lemma~2.3 of~\cite{ScZw18}!)
\item $\mathcal{B}_0$ has a linear bounded right-inverse $\mathcal{R}_0$, that is, $\mathcal{R}_0 \in L(\R^k,X)$ with $\mathcal{R}_0 \R^k \subset D(\mathcal{B}_0) = D(\mathcal{A}_0)$ and $\mathcal{B}_0 \mathcal{R}_0 u = u$ for all $u \in \R^k$
and, of course, $\mathcal{A}_0 \mathcal{R}_0 \in L(\R^k,X)$ (remark after Condition~3.1 of~\cite{ScZw18}!) 
\end{itemize}
With these observations at hand, we now see 
first that the operators
\begin{align*}
A(t) := \mathcal{A}(t)|_{\ker \mathcal{B}(t)} = \mathcal{A}_0|_{\ker \mathcal{B}_0} M(t) = A_0 M(t)
\end{align*}
satisfy the assumptions from Lemma~\ref{lm:vereinf-vor-A(t)} and thus Condition~\ref{ass:1, bdry}(i)
and second that 
the operators
\begin{align*}
\mathcal{R}(t) := M(t)^{-1} \mathcal{R}_0
\end{align*}
satisfy Condition~\ref{ass:1, bdry}(ii). 
And finally, Condition~\ref{ass:1, bdry}(iii) is obviously satisfied by virtue of  our regularity assumptions on $\mathcal{P}_c$ and $\mathcal{R}_c$. 
\smallskip

As a second step, we show that Condition~\ref{ass:2, bdry} is satisfied with 
\begin{align}
V(t,x) := \frac{1}{2} \scprd{\ul{\mathcal{H}}(t)\ul{x},\ul{x}}_2 + \mathcal{P}_c(v_1) + \frac{1}{2}v_2^{\top} K_c v_2 
\qquad ((t,x) \in \R^+_0 \times X).
\end{align}
Indeed, $V \in C^1(\R^+_0\times X,\R^+_0)$ and for every $s \in \R^+_0$ and $(x_s,u) \in \mathcal{D}_s$ we see with $(\ul{x}(t),v_1(t),v_2(t)) := x(t) := x(t,s,x_s,u)$ and $y(t) := \mathcal{C}(t)x(t,s,x_s,u)$ that
\begin{align}
\frac{\d}{\d t} V(t,x(t)) 
&= \frac{1}{2}\langle \dot{\ul{\mathcal{H}}}(t)\ul{x}(t),\ul{x}(t) \rangle_2 + \scprd{\ul{\mathcal{H}}(t)\ul{x}(t),\ul{\mathcal{A}}(t)\ul{x}(t)}_2 \notag \\
&\qquad \qquad \qquad \qquad \,\quad + \nabla \mathcal{P}_c(v_1(t))^{\top} \dot{v}_1(t) + (K_c v_2(t))^{\top} \dot{v}_2(t) \notag\\
&\le (\ul{\mathcal{B}}(t)\ul{x}(t))^{\top} \ul{\mathcal{C}}(t) \ul{x}(t) + (B_c^{\top}K_c v_2(t))^{\top}  \ul{\mathcal{C}}(t) \ul{x}(t) - (K_c v_2(t))^{\top} \mathcal{R}_c(t,K_c v_2(t)) \notag\\
&\le \big(\mathcal{B}(t)x(t) - S_c \ul{\mathcal{C}}(t)\ul{x}(t)\big)^{\top} \ul{\mathcal{C}}(t) \ul{x}(t) \le u(t)^{\top}y(t) - \varsigma |y(t)|^2
\end{align}
for all $t \in [s,T_{s,x_s,u})$. (In the first inequality, we used the monotonicity and impedance-passivity assumption~\eqref{eq:H(t)-mon-decreasing} and~\eqref{eq:open-loop-imp-passive-with-H=1}, in the second inequality we used the damping assumption on $\mathcal{R}_c$, and in the last inequality we used~\eqref{eq:S_c pos def}.)
Consequently, Condition~\ref{ass:2, bdry}(i) is satisfied with $\alpha := \frac{1}{2\varsigma}$ and $\beta := \frac{\varsigma}{2}$.
And in view of~\eqref{eq:H(t)(zeta)-lower-and-upper-bounded} and our assumptions on $\mathcal{P}_c$, Condition~\ref{ass:2, bdry}(ii) is satisfied as well (invoke  Lemma~2.5 of~\cite{ClLeSt98}, for instance, to see that $\mathcal{P}_c$ is equivalent to the 
norm $|\cdot|$ of $\R^{m_c}$).  
\smallskip

As a third step, we show that Condition~\ref{ass:3, bdry} is satisfied. Condition~\ref{ass:3, bdry}(i) is obviously satisfied and in order to verify 
Condition~\ref{ass:3, bdry}(ii) we use that the graph norm $\norm{\cdot}_{\ul{\mathcal{A}}_0} := \norm{\cdot}_{\ul{X}} + \norm{\ul{\mathcal{A}}_0 \cdot}_{\ul{X}}$ 
of the port-Hamiltonian operator $\ul{\mathcal{A}}_0$ is equivalent to the norm of $H^N((a,b),\R^m)$, that is, for some $\ul{c},\ol{c} \in (0,\infty)$
\begin{align*}
\ul{c}\norm{\ul{f}}_{H^N((a,b),\R^m)} \le \norm{\ul{f}}_{\ul{\mathcal{A}}_0} \le \ol{c} \norm{\ul{f}}_{H^N((a,b),\R^m)} 
\qquad (\ul{f} \in D(\ul{\mathcal{A}}_0))
\end{align*}
(Lemma~3.2.3 of~\cite{Augner}). With this equivalence of norms and the continuous embedding $H^N((a,b),\R^m) \hookrightarrow C^{N-1}([a,b],\R^m)$, we get for every $s \in \R^+_0$ and $(x_s,u) \in \mathcal{D}_s$ with $(\ul{x}(t),v_1(t),v_2(t)) := x(t) := x(t,s,x_s,u)$ that
\begin{align} \label{eq:cl-output-continuous,bdry}
\big| (\ul{\mathcal{H}}(t)\ul{x}(t))|_{\partial} &- (\ul{\mathcal{H}}(t_0)\ul{x}(t_0))|_{\partial} \big| 
\le 
\norm{\ul{\mathcal{H}}(t)\ul{x}(t)- \ul{\mathcal{H}}(t_0)\ul{x}(t_0)}_{C^{N-1}([a,b],\R^m)} \notag \\
&\le 
C \norm{\ul{\mathcal{H}}(t)\ul{x}(t)- \ul{\mathcal{H}}(t_0)\ul{x}(t_0)}_{\ul{\mathcal{A}}_0} \notag \\
&= 
C \norm{\ul{\mathcal{H}}(t)\ul{x}(t)- \ul{\mathcal{H}}(t_0)\ul{x}(t_0)}_{\ul{X}} + C \norm{\ul{\mathcal{A}}(t)\ul{x}(t)- \ul{\mathcal{A}}(t_0)\ul{x}(t_0)}_{\ul{X}}
\end{align}
for all $t, t_0 \in [s,T_{s,x_s,u})$. Since $t \mapsto \ul{x}(t) \in \ul{X}$ and $t \mapsto \ul{\mathcal{A}}(t)\ul{x}(t) = \dot{\ul{x}}(t) \in \ul{X}$ are continuous and $t \mapsto \ul{\mathcal{H}}(t)$ is strongly continuous, it follows from~\eqref{eq:cl-output-continuous,bdry} that the classical output
\begin{align*}
t \mapsto y(t,s,x_s,u) = \ul{\mathcal{C}}(t)\ul{x}(t) = W_C (\ul{\mathcal{H}}(t)\ul{x}(t))|_{\partial}
\end{align*}
is continuous and hence measurable, as desired.  
\end{proof}

\begin{ex} \label{ex:vibrating-string}
Consider a vibrating string with possibly time-dependent material coefficients~\cite{JaLa19}, that is, the transverse displacement $w(t,\zeta)$ of the string at 
position $\zeta \in (a,b)$ evolves according to the partial differential equation
\begin{align} \label{eq:string pde}
\partial_t \big( \rho(t,\zeta) \partial_t  w(t,\zeta) \big) =  \partial_{\zeta} \big( T(t,\zeta) \partial_{\zeta}w(t,\zeta) \big)
\end{align}
for $t \in \R^+_0$ and $\zeta \in (a,b)$ (vibrating string equation). 
In these equations, the material coefficients $\rho$, $T$ are the mass density and the Young modulus of the string, respectively. We assume that for some $\ul{m},\ol{m} \in (0,\infty)$ 
\begin{align*}
\ul{m} \le \rho(t,\zeta), T(t,\zeta) \le \ol{m} \qquad ((t,\zeta) \in  \R^+_0 \times (a,b)),
\end{align*}
that for $l = 0,1,2$ the partial derivatives $\partial_t^l \rho$, $\partial_t^l T$ exist and are continuous on $\R^+_0 \times (a,b)$ and that $t \mapsto \rho(t,\zeta)$ is monotonically increasing while $t \mapsto T(t,\zeta)$ is monotonically decreasing for every $\zeta \in (a,b)$. 
Also, assume that the string is clamped at its left end, that is, 
\begin{align} \label{eq:string bdry cond}
\partial_t w(t,a) = 0 \qquad (t \in \R^+_0)
\end{align}  
and that the control input $\ul{u}(t)$ and observation output $\ul{y}(t)$ are given respectively by the force and by the velocity at the right end of the string, that is,
\begin{align} \label{eq:string input/output}
\ul{u}(t) = T(t,b) \partial_{\zeta} w(t,b)
\qquad \text{and} \qquad
\ul{y}(t) = \partial_t w(t,b)
\end{align}
for all $t \in \R^+_0$. With the choices
\begin{align*}
\ul{x}(t)(\zeta) 
:=
\begin{pmatrix}
\rho(t,\zeta) \partial_t w(t,\zeta) \\ \partial_{\zeta} w(t,\zeta)
\end{pmatrix},
\qquad 
\ul{\mathcal{H}}(t,\zeta) 
:=
\begin{pmatrix}
1/\rho(t,\zeta) & 0 \\ 0 & T(t,\zeta)
\end{pmatrix},
\qquad
P_1 := \begin{pmatrix} 0 & 1 \\ 1 & 0 \end{pmatrix}
\end{align*}
and $P_0 := 0 \in \R^{2\times 2}$, the pde~\eqref{eq:string pde} takes the form~\eqref{eq:open-loop-diff-eq} of a port-Hamiltonian system of order $N =1$ and, moreover, the boundary condition~\eqref{eq:string bdry cond} and the in- and output conditions~\eqref{eq:string input/output} take the desired form~\eqref{eq:open-loop-domain} and~\eqref{eq:open-loop-input/output-eq}, 
with matrices $W_{B,1}, W_{B,2}, W_C \in \R^{1\times 4}$. 
It is straightforward to verify that the impedance-passivity condition~\eqref{eq:open-loop-imp-passive-with-H=1} is satisfied, that the matrix $W \in \R^{3\times 4}$ from~\eqref{eq:W-matrix, def} has full rank, and that all the assumptions on $\ul{\mathcal{H}}$, especially the bounds~\eqref{eq:H(t)(zeta)-lower-and-upper-bounded} and the monotonicity~\eqref{eq:H(t)-mon-decreasing}, are satsified. 
So, as soon as the controller is chosen as in Section~\ref{sect:appl-bdry-open-loop} above, the resulting closed-loop system will be well-posed and uniformly globally stable by Corollary~\ref{cor:appl-bdry}.  
~$\blacktriangleleft$
\end{ex}

\begin{ex} \label{ex:Timoshenko-beam}
Consider a Timoshenko beam with possibly time-dependent material coefficients~\cite{JaLa19}, that is, the transverse displacement $w(t,\zeta)$ and the rotation angle $\phi(t,\zeta)$ of the beam at 
position $\zeta \in (a,b)$ evolve according to the partial differential equations
\begin{gather} 
\partial_t \big( \rho(t,\zeta) \partial_t w(t,\zeta) \big) = \partial_{\zeta} \Big( K(t,\zeta) \big( \partial_{\zeta}w(t,\zeta) - \phi(t,\zeta) \big) \Big) \label{eq:Timoshenko pde,1}\\
\partial_t \big( I_r(t,\zeta) \partial_t \phi(t,\zeta) \big) = \partial_{\zeta} \big( E I(t,\zeta) \partial_{\zeta} \phi(t,\zeta) \big) + K(t,\zeta) \big( \partial_{\zeta}w(t,\zeta) - \phi(t,\zeta) \big) \label{eq:Timoshenko pde,2}
\end{gather}
for $t \in \R^+_0$ and $\zeta \in (a,b)$ (Timoshenko beam equations). In these equations, $\rho$, $E$, $I$, $I_r$, $K$ are the mass density, the Young modulus, the moment of inertia, the rotatory moment of inertia, and the shear modulus of the beam, respectively. 
We assume that for some $\ul{m},\ol{m} \in (0,\infty)$ 
\begin{align*}
\ul{m} \le \rho(t,\zeta), E I(t,\zeta), I_r(t,\zeta), K(t,\zeta) \le \ol{m} \qquad ((t,\zeta) \in  \R^+_0 \times (a,b)),
\end{align*}
that for $l = 0,1,2$ the partial derivatives $\partial_t^l \rho$, $\partial_t^l E I$, $\partial_t^l I_r$, $\partial_t^l K$ exist and are continuous on $\R^+_0 \times (a,b)$ and that $t \mapsto \rho(t,\zeta), I_r(t,\zeta)$ are monotonically increasing while $t \mapsto EI(t,\zeta), K(t,\zeta)$ are monotonically decreasing for every $\zeta \in (a,b)$. 
Also, assume that the beam is clamped at its left end, that is, 
\begin{align} \label{eq:Timoshenko bdry cond}
\partial_t w(t,a) = 0 \qquad \text{and} \qquad \partial_t \phi(t,a) = 0
\qquad (t \in \R^+_0)
\end{align}  
(velocity and angular velocity at the left endpoint $a$ are zero), and that the control input $\ul{u}(t)$ is given by the force and the torsional moment at  the right end of the beam and the observation output $\ul{y}(t)$ is given by the velocity and angular velocity at the right end of the beam, that is,
\begin{align} \label{eq:Timoshenko input/output}
\ul{u}(t) = \begin{pmatrix} K(t,b) \big( \partial_{\zeta} w(t,b) - \phi(t,b) \big) \\ E I(t,b) \partial_{\zeta} \phi(t,b) \end{pmatrix},
\qquad
\ul{y}(t) = \begin{pmatrix} \partial_t w(t,b) \\ \partial_t \phi(t,b) \end{pmatrix}
\end{align}
for all $t \in \R^+_0$. With the choices
\begin{align*}
x(t)(\zeta) 
:=
\begin{pmatrix}
\partial_{\zeta} w(t,\zeta) - \phi(t,\zeta) \\
\rho(t,\zeta) \partial_t w(t,\zeta) \\
\partial_{\zeta} \phi(t,\zeta) \\
I_r(t,\zeta) \partial_t \phi(t,\zeta)
\end{pmatrix},
\quad 
\mathcal{H}(\zeta) 
:=
\begin{pmatrix}
K(t,\zeta) & 0 & 0 & 0 \\
0 & 1/\rho(t,\zeta) & 0 & 0 \\ 
0 & 0 & EI(t,\zeta) & 0 \\
0 & 0 & 0 & 1/I_r(t,\zeta)
\end{pmatrix},
\end{align*}
and the same choice of $P_1, P_0 \in \R^{4\times 4}$ as in~\cite{JacobZwart}, the pde~\eqref{eq:Timoshenko pde,1}-\eqref{eq:Timoshenko pde,2} take the form~\eqref{eq:open-loop-diff-eq} of a port-Hamiltonian system of order $N =1$ and, moreover, the boundary condition~\eqref{eq:Timoshenko bdry cond} and the in- and output conditions~\eqref{eq:Timoshenko input/output} take the desired form~\eqref{eq:open-loop-domain} and~\eqref{eq:open-loop-input/output-eq} with matrices $W_{B,1}, W_{B,2}, W_C \in \R^{2\times 8}$. 
It is straightforward to verify that impedance-passivity condition~\eqref{eq:open-loop-imp-passive-with-H=1} is satisfied, that the matrix $W \in \R^{6\times 8}$ from~\eqref{eq:W-matrix, def} has full rank, and that all the assumptions on $\ul{\mathcal{H}}$, especially the bounds~\eqref{eq:H(t)(zeta)-lower-and-upper-bounded} and the monotonicity~\eqref{eq:H(t)-mon-decreasing}, are satsified. 
So, as soon as the controller is chosen as in Section~\ref{sect:appl-bdry-open-loop} above, the resulting closed-loop system will be well-posed and uniformly globally stable by Corollary~\ref{cor:appl-bdry}.~$\blacktriangleleft$
\end{ex}

\section*{Acknowledgements}

I would like to thank Hafida Laasri for interesting discussions and the German Research Foundation (DFG) for financial support through the grant DA 767/7-1. 

\begin{small}

\end{small}

\end{document}